\documentclass{amsart}
\usepackage[utf8]{inputenc}

\usepackage{enumerate, amsmath, amsfonts, amssymb, amsthm, wasysym, xcolor, url, hyperref, hypcap, frcursive, longtable, caption}
\hypersetup{colorlinks=true, citecolor=blue, linkcolor=blue}

\usepackage{tikz}
\usetikzlibrary{scopes,positioning}
\tikzset{edot/.style={circle,draw=black,fill=black,minimum size=5pt,inner sep=0}}
\tikzset{ndot/.style={circle,draw=black,fill=white,minimum size=12pt,inner sep=1}}
\tikzset{cdot/.style={circle,draw=black,fill=black,minimum size=1pt,inner sep=0}}
\usetikzlibrary{decorations.pathmorphing}

\newtheorem{propo}{Proposition}[section]

\newtheorem{theor}[propo]{Theorem}

\newtheorem{conje}[propo]{Conjecture}
\newtheorem{algor}[propo]{Algorithm}

\theoremstyle{definition}

\newtheorem{property}{Property}

\theoremstyle{remark}
\newtheorem{remar}[propo]{Remark}

\numberwithin{equation}{section}

\newcommand{\algo}[6]
{
\begin{algor}{{\tt #1}{\rm (#2)}}\label{#1}\end{algor}
\vspace{-6pt}\noindent{\it #3}.

{\bf Input:} #4

{\bf Output:} #5

\newcounter{#1}
\begin{list}{\textbf{\arabic{#1}.}}{\usecounter{#1}}
#6\end{list}\vspace{3pt}}

\newcommand{\NN }{\mathbb{N}}
\newcommand{\CC }{\mathbb{C}}
\newcommand{\RR }{\mathbb{R}}
\newcommand{\FF }{\mathbb{F}}
\newcommand{\QQ }{\mathbb{Q}}
\newcommand{\ZZ }{\mathbb{Z}}

\DeclareMathOperator{\rank}{rk}

\newcommand{\degx}{\operatorname{d}_\x}
\newcommand{\degy}{\operatorname{d}_\y}

\newcommand{\PhD}{\Phi^+_{D_6}}
\newcommand{\PhH}{\Phi^+_{H_3}}

\newcommand{\defn}[1]{\emph{\color{red} #1}} 
\newcommand{\fref}[1]{Figure~\ref{#1}} 
\newcommand{\tref}[1]{Table~\ref{#1}} 

\def\rootSystem{\Phi} 
\def\rootPoset{\Phi^+} 
\def\posrootPoset{\Phi^+_{\sf res}} 
\def\simpleSystem{\Delta} 
\def\orbits{\mathcal{O}} 
\def\orbit{{\operatorname{O}}} 
\def\bijection{{\Psi}} 
\def\antichains{\mathcal{A}} 
\def\orderideals{\mathcal{I}} 
\def\Cat{\operatorname{Cat}} 
\def\Pan{{\sf P}} 
\def\posPan{{\sf P}_{\sf res}} 
\def\Krew{{\sf K}} 
\def\NC{N\!C} 
\def\one{{\bf 1}} 
\def\coinv{\mathcal{M}} 
\def\x{{\bf x}} 
\def\y{{\bf y}} 
\def\hilbert{{\mathcal{H}}} 

\newcommand{\multiset}[1]{\left\{\!\!\left\{ #1 \right\}\!\!\right\}} 

\title[On root posets for noncrystallographic root systems]{On root posets for\\noncrystallographic root systems}

\author{Michael~Cuntz}
\thanks{Most of the results of this article were achieved at the Leibniz Universit\"at Hannover in summer 2012.}
\address{Fachbereich Mathematik, Universit\"at Kaiserslautern}
\email{cuntz@mathematik.uni-kl.de}
\urladdr{http://www.mathematik.uni-kl.de/~cuntz}

\author{Christian~Stump}
\address{Institut f\"ur Algebra, Zahlentheorie, Diskrete Mathematik, Universit\"at Hannover}
\email{stump@math.uni-hannover.de}
\urladdr{http://homepage.univie.ac.at/christian.stump/}

\thanks{}

\begin{document}

\begin{abstract}
We discuss properties of root posets for finite crystallographic root systems, and show that these properties uniquely determine root posets for the noncrystallographic dihedral types and type $H_3$, while proving that there does not exist a poset satisfying all of the properties in type $H_4$. We do this by exhaustive computer searches for posets having these properties. We further give a realization of the poset of type $H_3$ as restricted roots of type $D_6$, and conjecture a Hilbert polynomial for the $q,t$-Catalan numbers for type $H_4$.
\end{abstract}

\maketitle

\tableofcontents

\section{Introduction}

Let $\rootSystem$ be a finite root system with simple roots $\simpleSystem$ and positive roots $\rootPoset$.
For crystallographic root systems, the \defn{root poset} is given by the partial order on positive roots defined as
$$\beta \leq \beta' \quad\text{if}\quad \beta' - \beta \in \NN\Delta.$$
Multiple enumerative properties of this poset and in particular about the collection of antichains in this poset can be deduced from the degrees of $\rootSystem$, which are defined to be the degrees of the fundamental invariants for the Weyl group $W(\rootSystem)$ associated to $\rootSystem$. We collect properties of root posets below in Properties~\ref{prop:parabolicsubsystems} through~\ref{prop:antichainorderidealsize}. Even though all the enumerative information can be studied as well for noncrystallographic root systems, the analogous posets in these cases with $\NN$ in the definition being replaced by $\RR_{\ge 0}$ do not have the desired properties.

This led D.~Armstrong to raise the question whether there exist posets for the noncrystallographic types having these properties. He moreover positively answered the question for the noncrystallographic dihedral types and for type $H_3$, see~\cite[Section~5.4.1]{Arm2006}. Even though he did not take Properties~\ref{prop:antichainorbitsize} and~\ref{prop:antichainorderidealsize} into account, the posets he constructed do indeed fulfill also those.

This leaves type $H_4$ as the only type remaining.
The main purpose of this paper is to discuss this type in full generality.
\begin{theor}
\label{th:maintheorem}
  Possible root posets for noncrystallographic types are given as follows:
  \begin{enumerate}[(i)]
    \item The posets in \fref{fig_posetH3} on page~\pageref{fig_posetH3} are the unique posets in types $I_2(m)$ and $H_3$ satisfying Properties~\ref{prop:parabolicsubsystems} through~\ref{prop:antichainorderidealsize}.
    \label{item1}

    \item There are billions of posets in type $H_4$ satisfying Properties~\ref{prop:parabolicsubsystems} through~\ref{prop:antichainsize}.
    \label{item2}
    
    \item None of these posets in type $H_4$ satisfy Property~\ref{prop:antichainorbitsize}.
    \label{item3}

    \item None of these posets in type $H_4$ satisfy Property~\ref{prop:antichainorderidealsize}.
    \label{item4}
  \end{enumerate}
\end{theor}

\begin{remar}
  We want to remark that even in the crystallographic types, Property~\ref{prop:antichainsize}, Property~\ref{prop:antichainorbitsize}(b), and Property~\ref{prop:antichainorderidealsize} are partially conjectured and remain open in full generality.
\end{remar}

\begin{remar}
  In~\cite{CK2007}, Y.~Chen and C.~Kriloff studied properties of root posets in noncrystallographic types in the sense of replacing $\NN$ by $\RR_{\geq 0}$ in the definition as mentioned above.
  Their study has a somewhat different flavor since they study the correspondence between antichains in root posets and region in Shi arrangements.
  As also mentioned above, these posets do not have the desired properties in the sense we discuss here.
  Nevertheless, the authors obtain, similarly to our results, well behaved connections between antichain and these regions in types $I_2(m)$ and $H_3$, while they show that the connections in type $H_4$ are not well behaved.
\end{remar}

We prove Theorem~\ref{th:maintheorem} by an exhaustive computer search for posets satisfying the various properties.
We used (and present here) two different algorithms for the computations.
In both algorithms we construct posets inductively. The breakthrough is achieved by testing the right properties at the right moment.

\bigskip

In Section~\ref{sec:rootposets}, we review the combinatorial properties of root posets for finite crystallographic root systems.
In Section~\ref{sec:H3}, we prove Theorem~\ref{th:maintheorem}\eqref{item1}.
Section~\ref{sec:H4} is devoted to provide two algorithmic approaches to prove Theorem~\ref{th:maintheorem}\eqref{item2}--\eqref{item4}.
In the final Section~\ref{sec:hilbertseries}, we describe how we obtained the conjectured Hilbert series needed to provide Property~\ref{prop:antichainorderidealsize} for type $H_4$.

\section{Root posets and their combinatorics}
\label{sec:rootposets}

Throughout this paper, let $\rootSystem$ be an irreducible and finite root system of rank~$n$ with simple roots~$\simpleSystem$ and positive roots~$\rootPoset$. Moreover, let $W = W(\rootSystem)$ be the corresponding reflection group acting on a real vector space $V$, $c \in W$ be a Coxeter element, $d_1 \geq \dots \geq d_n$ be the degrees for $W$, and $h = d_1$ its Coxeter number.
We refer to~\cite{Hum1990} for background on finite root systems and reflection groups.
It is well known that irreducible finite root systems are classified according to Cartan-Killing types

\begin{itemize}
  \item $A_n$ with $n \geq 1$,
  \item $B_n$ with $n \geq 2$,
  \item $C_n$ with $n \geq 3$,
  \item $D_n$ with $n \geq 4$,
  \item $G_2,F_4,E_6,E_7,E_8$,
  \item $H_3,H_4$, and $I_2(m)$ with $m = 5$ or $m \geq 7$.
\end{itemize}
A root system $\rootSystem$ is \defn{crystallographic} if the group $W(\rootSystem)$ stabilizes a lattice in $V$.
It is well known that the only noncrystallographic finite root systems are those of types $H$ and $I$.

\medskip

For a crystallographic root system, define the \defn{root poset} $\rootPoset = (\rootPoset,\leq)$ to be given by the partial order $\beta \leq \beta' :\Leftrightarrow \beta' - \beta \in \NN\Delta$ on positive roots.
This poset is \defn{graded}, i.e. all maximal chains in $\rootPoset$ have the same length. The rank of a positive root $\beta = \sum_{\alpha \in \simpleSystem} \lambda_\alpha \alpha$ is given by $\sum_{\alpha \in \simpleSystem} \lambda_\alpha$, and it is known that the maximal rank is equal to $h-1$.
The following definitions make sense for any poset, though we discuss them here only in the context of root posets. For a set $X \subseteq \rootPoset$, we denote by $\min(X)$ and by $\max(X)$ the set of all minimal and maximal positive roots in $X$, respectively.
An \defn{antichain} in $\rootPoset$ is a set of positive roots of mutually incomparable elements. We denote the collection of all antichains in $\rootPoset$ by
$$\antichains(\rootSystem) = \big\{ A \subseteq \rootPoset : A \text{ antichain} \big\}.$$
Moreover, an \defn{order ideal} in $\rootPoset$ is a set $I$ of positive roots such that
$$ \beta \in I,\:\: \gamma \leq \beta \quad\Longrightarrow\quad \gamma \in I,$$
and we define as well $\orderideals(\rootSystem)$ to be the collection of all order ideals in $\rootPoset(\rootSystem)$.
Obviously, antichains and order ideals are in one-to-one-correspondence by sending
\begin{align*}
  A \mapsto I(A) &:= \big\{ \gamma \in \rootPoset : \gamma \leq \beta \text{ for some } \beta \in A \big\}, \\
  I \mapsto A(I) &:= \max(I).
\end{align*}

We are now in the position to discuss the enumerative properties of antichains (or order ideals) in root posets for crystallographic root systems. In Property~\ref{prop:parabolicsubsystems}, we collect some direct consequences of the definition, Property~\ref{prop:exponentsinranks} was discussed e.g.\ in~\cite[Theorem~3.20]{Hum1990}, Properties~\ref{prop:antichains} and~\ref{prop:antichainsize} can be found in~\cite[Section~5.4.1]{Arm2006}, while Properties~\ref{prop:antichainorbitsize} and~\ref{prop:antichainorderidealsize} are taken from~\cite{AST2010} and from~\cite{Stu2010}, respectively.
All these enumerative properties can as well be studied for noncrystallographic root systems, and each of them can be used to find possible root posets in these types.

\begin{property}[Basic properties of root systems]
\label{prop:parabolicsubsystems}
This first property is rather a collection of \lq\lq natural\rq\rq\ properties of root posets that are satisfied by definition.
\begin{enumerate}[1.]
  \item The minimal elements in $\rootPoset$ correspond to the simple system $\simpleSystem$ for the root system $\rootSystem$.
  \item There is a unique maximal element in the poset $\rootPoset$.
  \item The poset $\rootPoset$ is graded.
  \item The poset $\rootPoset$ behaves well with respect to standard parabolic subsystems. This is, the subposet of $\rootPoset$ given by all $\beta \in \rootPoset$ which lie in a subspace generated by a subset $\simpleSystem'$ of the simple system $\simpleSystem$ is equal to the root poset for the root system $\rootSystem'$ generated by $\simpleSystem'$. \label{eq:parabolic}
\end{enumerate}
\end{property}

\begin{property}[The degrees of a root system]
\label{prop:exponentsinranks}
The second property describes how the degrees of a root system are encoded in the root poset, or rather in the rank generating function of the positive roots.
It is taken from~\cite[Theorem~3.20]{Hum1990}.
\begin{theor}
  Let $d_1 \geq \ldots \geq d_n$ denote the degrees of $\rootSystem$. Then the number $\rank_i$ of positive roots in $\rootPoset$ of rank $i$ is given by
  $$\rank_i = \Big|\big\{ j : d_j > i \big\} \Big|.$$
  In particular, the maximal rank in $\rootPoset$ is given by $h-1$.
\end{theor}
The ranks $\rank(\rootPoset) = (\rank_1,\dots,\rank_{h-1})$ of a noncrystallographic root poset $\rootPoset$ satisfying this property are collected in \tref{tab:ranks}, where we write $i^j$ for $j$ consecutive ranks of size $i$.
\begin{table}[t]
  \centering
  \begin{tabular}[h]{r|c|c}
          & $\rank(\rootPoset)$ & $|\antichains(\rootPoset)|$ \\
    \hline
    \hline
    &&\\[-10pt]
    $I_2(m)$ & $(2,1^{m-2})$ & $m+2$ \\
    \hline
    &&\\[-10pt]
    $H_3$ & $(3,2^4,1^4)$ & $32$ \\
    \hline
    &&\\[-10pt]
    $H_4$ & $(4,3^{10},2^8,1^{10})$ & $280$ \\
  \end{tabular}
  \caption{The desired ranks and desired numbers of antichains.}
  \label{tab:ranks}
\end{table}
\end{property}

\begin{property}[Counting antichains]
\label{prop:antichains}
The \defn{$\rootSystem$-Catalan numbers} are defined as
$$\Cat(\rootSystem) = \prod_{i = 1}^n \frac{d_i + h}{d_i}.$$
The following theorem is due to A.~Postnikov, see~\cite[Remark~2]{Rei1997}.
\begin{theor}
  Antichains (or equivalently order ideals) in the root poset are counted by the $\rootSystem$-Catalan numbers,
  $$\big| \antichains(\rootSystem) \big| = \Cat(\rootSystem).$$
\end{theor}
\end{property}
This theorem yields the desired numbers of antichain for root posets of noncrystallographic types as collected in \tref{tab:ranks}.

\bigskip

The following two properties (Properties~\ref{prop:antichainsize} and~\ref{prop:antichainorbitsize}) are both strictly stronger than this property.
This is to say that both provide in particular the above counting formula for antichains.

\begin{property}[Counting antichains according to size]
\label{prop:antichainsize}
This property concerns the number of antichains of a particular size containing a particular number of simple roots. These numbers are encoded in the following generating function
$$H(\rootSystem;s,t) = \sum_{A \in \antichains(\rootSystem)} s^{|A \cap \simpleSystem|} t^{|A|}.$$

There is a remarkable conjectured connection between this and two other generating functions. These encode information about the noncrossing partition lattice and about the cluster complex. The three generating functions were defined by F.~Chapoton in~\cite{Cha2004,Cha2006} and are called \defn{Chapoton's H-triangle}, \defn{M-triangle}, and \defn{F-triangle}.
We restrict our attention here to the $H$- and the $M$-triangle. For a detailed treatment of all three, see Chapoton's original papers~\cite{Cha2004,Cha2006} and as well~\cite[Section~5]{Arm2006}.

The M-triangle encodes information about the Möbius function on the \defn{noncrossing partition lattice} $\NC(\rootSystem)$ given by all elements in the interval $[\one,c]$ in $W$ endowed with the absolute order where $c$ denotes a Coxeter element in $W$. It is defined by
$$M(\rootSystem;x,y) = \sum_{\substack{\sigma,\pi \in \NC(\rootSystem) \\ \sigma \leq \pi}} \mu(\sigma,\pi) x^{n-\rank(\sigma)} y^{n-\rank(\pi)},$$
where $\mu:\NC(\rootSystem) \times \NC(\rootSystem) \longrightarrow \ZZ$ denotes the Möbius function on the lattice $\NC(\rootSystem)$.
Observe that we dropped the Coxeter element $c$ from the notations. This is due to the fact that the resulting lattice is independent of the chosen Coxeter element.
F.~Chapoton conjectured deep numerical interactions between antichains in the root poset and the Möbius function on the noncrossing partition lattice. This conjecture remains open.
\begin{conje}
\label{con:chapoton}
  Chapoton's $H$-triangle is related to the $M$-triangle via the identity
  \begin{align*}
    H(\rootSystem;s,t) &= (1 + (s-1)t)^n M\left(\rootSystem; \frac{s}{s-1},\frac{(s-1)t}{1+(s-1)t} \right).
  \end{align*}
\end{conje}
\end{property}

\begin{table}[t]
  \centering
  \begin{tabular}{c | l}
    & $H(\rootSystem;s,t)$ \\
    \hline
    \hline
    &\\[-10pt]
    $I_2(m)$ & \small $1 + 2st + (m-2)t + s^2t^2$ \\
    \hline
    &\\[-10pt]
    $H_3$ & \small $1 + 3st + 12t + 3s^2t^2 + 4st^2 + 8t^2 + s^3t^3$ \\
    \hline
    &\\[-10pt]
    $H_4$ & \small $1 + 4st + 56t+ 6s^2t^2 + 19st^2 + 133t^2 + 4s^3t^3 + 5s^2t^3 + 9 st^3 + 42t^3 + s^4t^4$
  \end{tabular}
  \caption{The desired $H$-triangles.}
  \label{tab:triangle}
\end{table}

Conjecture~\ref{con:chapoton} can now be used to determine the sizes of the antichains in noncrystallographic root posets satisfying this property. The corresponding generating functions are collected in \tref{tab:triangle}. They are taken from~\cite[Figure~5.14]{Arm2006} and independently verified using both the $M$- and the $F$-triangle.

\begin{property}[Counting antichains according to Panyushev orbits]
\label{prop:antichainorbitsize}

This property concerns a cyclic action on antichains in the root poset defined by D.I.~Panyushev in~\cite{Pan2008}. He conjectured multiple numerical results which were then later proven by D.~Armstrong, H.~Thomas, and the second author in~\cite{AST2010}.
The \defn{Panyushev map} $\Pan : \antichains(\rootSystem) \longrightarrow \antichains(\rootSystem)$ is given by mapping an antichain $A$ to the collection of minimal elements in the root poset among all elements that are not in the order ideal generated by~$A$,
$$\Pan(A) := \min \big( \rootPoset \setminus I(A) \ \big).$$

It is shown in~\cite{AST2010} that the orbit structure of $\Pan$ on $\antichains(\rootPoset)$ coincides with another remarkable action on the noncrossing partition lattice given by the \defn{Kreweras map} $\Krew : \NC(\rootSystem) \longrightarrow \NC(\rootSystem)$ which is defined by
$$\Krew(\sigma) = \sigma^{-1}c.$$
This was done by constructing an equivariant bijection $\bijection : \antichains(\rootSystem) \tilde \longrightarrow \NC(\rootSystem)$.
We denote by $\orbits(\rootSystem)$ the multiset of sizes of Kreweras orbits in $\NC(\rootSystem)$,
$$\orbits(\rootSystem) = \multiset{ | \orbit | : \orbit \subseteq W \text{ is a Kreweras orbit} }.$$
The multisets of Kreweras orbit sizes for the noncrystallographic types are collected in \tref{tab:orbits}. They should coincide with the Panyushev orbit sizes of a noncrystallographic root poset satisfying this property.
\begin{table}[t]
  \centering
  \begin{tabular}{c | l}
    & $\orbits(\rootSystem)$ \\
    \hline
    \hline
    &\\[-10pt]
    $I_2(m)$ & $\multiset{2,m}$ \\
    \hline
    &\\[-10pt]
    $H_3$ & $\multiset{2,10^3}$ \\
    \hline
    &\\[-10pt]
    $H_4$ & $\multiset{2,3,5,30^9}$
  \end{tabular}
  \caption{The desired multisets of Panyushev orbit sizes.}
  \label{tab:orbits}
\end{table}

\medskip

The following extensions of this property complete the collection of (partially conjectured) properties of root posets.

\bigskip

\noindent{\bf Property~5(a).} 
As conjectured in~\cite{Pan2008} and proven in~\cite{AST2010}, the average number of positive roots in an antichain in a Panyushev orbit is constant,
$$\frac{1}{|\orbit|}\sum_{A \in \orbit}|A| = \frac{n}{2},$$
where $\orbit$ is any Panyushev orbit in $\antichains(\rootSystem)$, and $n$ is the rank of $\rootSystem$.

\bigskip

\noindent{\bf Property~5(b).}
The previous property has a counterpart on the average number of positive roots in an antichain in a \lq\lq restricted Panyushev orbit\rq\rq. Let $\posrootPoset$ be the poset obtained from the root poset $\rootPoset$ by deleting the simple roots, and let $\posPan : \antichains(\posrootPoset) \longrightarrow \antichains(\posrootPoset)$ be the \defn{restricted Panyushev map} defined in the same way as the Panyushev map $\Pan: \antichains(\rootPoset) \longrightarrow \antichains(\rootPoset)$ with the poset $\rootPoset$ replaced by $\posrootPoset$.
It is conjectured in~\cite{Pan2008} that the average number of positive roots in an antichain in a restricted Panyushev orbit is as well constant,
$$\frac{1}{|\orbit|}\sum_{A \in \orbit}|A| = \frac{n(h-2)}{2(h-1)},$$
where $\orbit$ is any restricted Panyushev orbit in $\antichains(\posrootPoset)$, $n$ is again the rank of $\rootSystem$, and where $h$ is the Coxeter number. Observe here that this conjecture holds true for simple reasons that can be deduced from~\cite{AST2010} in types $A$ and $B$, and was checked in the exceptional crystallographic types. Thus, only type $D$ remains open, and we expect that this type can as well be deduced from the results in~\cite{AST2010}.

\end{property}

\begin{property}[Counting antichains according to order ideal sizes]
\label{prop:antichainorderidealsize}
The last --~again conjectured~-- property comes from the theory of $q,t$-Catalan numbers. Those were extensively studied in type $A$ by various authors. We refer to~\cite{GH1996,Hag2008} and the references therein for detailed background and definitions.

Let $W = W(\rootSystem)$ act diagonally on $\CC[\x,\y] = \CC[x_1,\ldots,x_n,y_1,\ldots,y_n] = \CC[V \oplus V]$, let $\mathcal{J}$ be the ideal in $\CC[\x,\y]$ generated by all polynomials $f \in \CC[\x,\y]$ such that $w(f) = \det(w) f$ for all $w \in W$, and define the \defn{diagonal coinvariant ring} as
$$\coinv(\rootSystem) = \mathcal{J} \Big/ \langle \x,\y \rangle \mathcal{J}.$$
The \defn{$(\rootSystem;q,t)$-Catalan numbers} are defined to be the bigraded Hilbert series of the diagonal coinvariants (see Section \ref{sec:hilbertseries} for details),
$$ \Cat(\rootSystem;q,t) = \hilbert\big( \coinv(\rootSystem); q, t \big).$$
In the case of the symmetric group, it was proven by M.~Haiman in~\cite{Hai20022} that the specialization $t = 1$ in $\Cat(A_n;q,t)$ equals the size generating function of order ideals in the root poset of type $A_n$,
$$\Cat(A_n;q,1) = \sum_{I \in \orderideals(A_n)} q^{|I|}.$$
In~\cite{Stu2010}, the second author studied the $(\rootSystem;q,t)$-Catalan numbers for other reflection groups, and conjectured that this phenomenon describes the situation in general.
\begin{conje}
  The specialization $t=1$ in the $(\rootSystem;q,t)$-Catalan numbers yields the size generating function of order ideals in the root poset,
  $$\Cat(\rootPoset;q,1) = \sum_{I \in \orderideals(\rootPoset)} q^{|I|}.$$
\end{conje}
\begin{table}[t]
  \small
  \centering
  \begin{tabular}{c | l}
    & $\Cat(\rootSystem;q,t)$ \\
    \hline
    \hline
    &\\[-10pt]
    $I_2(m)$ & $[m+1] + qt$ \\
    \hline
    &\\[-10pt]
    $H_3$ & $[16] + qt[10] + qt[6]$ \\
    \hline
    &\\[-10pt]
    $H_4$ & $[61] + qt[49] + qt[41] + q^2t^2[37] + qt[31] + $ \\
          & $q^3t^3[25] + q^2t^2[21] + q^4t^4[13] + q^6t^6 + q^{10}t^{10}$
  \end{tabular}
  \caption{The (conjectured) $q,t$-Catalan numbers for the noncrystallographic root systems.}
  \label{tab:orderideals}
\end{table}

In \tref{tab:orderideals}, we provide the (conjectured) $(\rootSystem;q,t)$-Catalan numbers for the noncrystallographic types, where we write
\[ [n] = [n]_{q,t} := \frac{q^n-t^n}{q-t} = q^{n-1}+q^{n-2}t+\ldots+qt^{n-2}+t^{n-1}. \]
For the dihedral types $I_2(m)$ and type $H_3$, these Hilbert series can be directly computed using the definition, with the total degree of the polynomials $f \in \CC[\x,\y]$ being bounded by the number of positive roots.
The computations for type $H_4$ are again much more complicated since the brute force approach is currently far beyond the feasibility of modern computers.
We discuss these computations in Section~\ref{sec:hilbertseries}.

\end{property}

\section{The poset of dihedral types and type \texorpdfstring{$H_3$}{H3}}
\label{sec:H3}

\subsection{The dihedral types}

In type $I_2(m)$, one can easily check by hand that the poset shown in \fref{fig_posetH3} (left) is the unique poset satisfying Properties~\ref{prop:parabolicsubsystems} and~\ref{prop:exponentsinranks}. It as well satisfies Properties~\ref{prop:antichains} through~\ref{prop:antichainorderidealsize}.

\subsection{Type \texorpdfstring{$H_3$}{H3}}

Let $W$ be the Coxeter group of type $H_3$, i.e.\ the group with Coxeter graph:
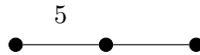
\begin{center}
\begin{tikzpicture}[baseline=4mm,x=12mm,y=4mm]
  \node[edot](1) at (1,2) {};
  \node[edot](2) at (2,2) {};
  \node[edot](3) at (3,2) {};
  \node(6) at (1.5,3) {5};
  \draw (1)--(2)(2)--(3);
\end{tikzpicture}
\end{center}

\begin{figure}
{\tiny

\begin{tikzpicture}[baseline=15mm,x=6mm,y=8mm]
\node[ndot](1) at (1,1) {1};
\node[ndot](2) at (3,1) {2};

\node[ndot](3) at (2,2) {3};
\node[ndot](4) at (2,3) {4};
\node[ndot](5) at (2,5) {m};

\begin{scope}[line width=0.8pt]
\draw (1) -- (3) (2) -- (3) (3) -- (4);
\draw[dashed] (4) -- (5);
\end{scope}
\end{tikzpicture}
\hspace*{50pt}
\begin{tikzpicture}[baseline=15mm,x=6mm,y=8mm]
\node[ndot](1) at (1,1) {1};
\node[ndot](2) at (3,1) {2};
\node[ndot](3) at (5,1) {3};

\node[ndot](5) at (2,2) {4};
\node[ndot](6) at (4,2) {5};
\node[ndot](8) at (2,3) {6};
\node[ndot](9) at (4,3) {7};
\node[ndot](11) at (2,4) {8};
\node[ndot](12) at (4,4) {9};
\node[ndot](14) at (2,5) {10};
\node[ndot](15) at (4,5) {11};
\node[ndot](17) at (3,6) {12};
\node[ndot](20) at (3,7) {13};
\node[ndot](23) at (3,8) {14};
\node[ndot](26) at (3,9) {15};

\begin{scope}[line width=0.8pt]
\draw (26)--(23) (23)--(20) (20)--(17) (17)--(14) (17)--(15) (14)--(11) (14)--(12) (15)--(12) (11)--(8)
(12)--(8) (12)--(9) (8)--(5) (9)--(5) (9)--(6) (5)--(1) (5)--(2) (6)--(2) (6)--(3);
\end{scope}
\end{tikzpicture}
}
\caption{The (unique) root posets of types $I_2(m)$ and $H_3$\label{fig_posetH3}}
\end{figure}
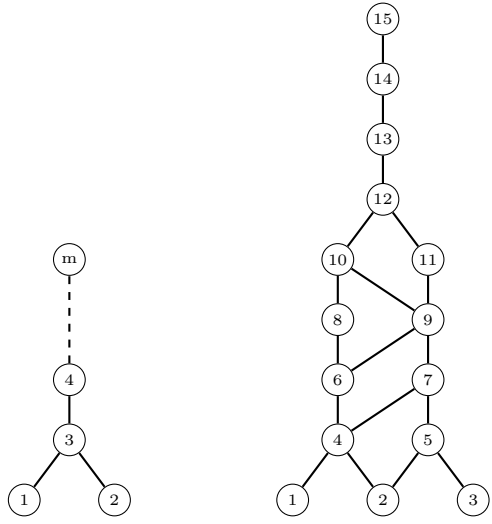

We use Algorithm~\ref{EnumeratePosets2} below to prove the uniqueness of a graded poset $(\Gamma=\{1,\ldots,15\},\prec)$ satisfying Properties~\ref{prop:parabolicsubsystems} through~\ref{prop:antichainorderidealsize}.

\medskip

Indeed, we have more in this type.
Namely, one can construct\footnote{Hugh Thomas informed us that he has a similar approach to the one we use here, in which one can construct the root poset of type $H_3$ as a subposet of the root poset of type $D_6$.} the root poset of type $H_3$ in a natural way from the root poset of type $D_6$.
Let $\tau=\frac{1+\sqrt{5}}{2}$ be the golden ratio.
Let $\PhD$ be the root poset for the root system corresponding to the Weyl group of type $D_6$ acting on the Euclidean space $\RR^6$ with standard basis $v_1,\ldots,v_6$ and bilinear form $(v_i\mid v_j)=\delta_{ij}$, see \cite[\S 6.7]{b-Kac90}. Then $\alpha_1=v_1-v_2,\ldots,\alpha_5=v_5-v_6,\alpha_6=v_5+v_6$ are the simple roots in $\PhD$.
\begin{figure}
\begin{tabular}{l l}
(0,0,0,0,1,1), & (1,-1,0,0,0,0),\\
(0,0,0,1,-1,0), & (0,1,-1,0,0,0),\\
(0,0,0,1,0,1), & (1,0,-1,0,0,0),\\
(0,0,1,-1,0,0), & (0,0,0,0,1,-1),\\
(0,0,1,0,0,-1), & (0,1,0,-1,0,0),\\
(0,0,1,0,1,0), & (1,0,0,-1,0,0),\\
(0,1,0,0,-1,0), & (0,0,0,1,0,-1),\\
(0,1,0,0,0,-1), & (0,0,1,0,-1,0),\\
(0,1,0,1,0,0), & (1,0,0,0,-1,0),\\
(0,1,1,0,0,0), & (1,0,0,0,0,-1),\\
(1,0,0,0,0,1), & (0,0,0,1,1,0),\\
(1,0,0,0,1,0), & (0,0,1,0,0,1),\\
(1,0,0,1,0,0), & (0,1,0,0,0,1),\\
(1,0,1,0,0,0), & (0,1,0,0,1,0),\\
(1,1,0,0,0,0), & (0,0,1,1,0,0).
\end{tabular}
\caption{The positive roots of the root system of type $D_6$\label{fig:posroD6}}
\end{figure}
With respect to the basis $v_1,\ldots,v_6$, the positive roots of $\PhD$ are listed in \fref{fig:posroD6}, where a pair $(\alpha,\beta)$ in a row is given as follows.
Write $v_i=\sum_j a_{i,j}\alpha_j$ for $a_{i,j}\in\ZZ$. With
\begin{align*}
\varepsilon :&\ \ZZ^6\rightarrow \ZZ^6, \quad v_i \mapsto (a_{i,1},\ldots,a_{i,6}), \quad \text{(base change)}\\
\gamma :&\ \ZZ^6\rightarrow \ZZ[\tau]^3, \quad (a_1,a_2,a_3,a_4,a_5,a_6) \mapsto (a_1+a_2\tau,a_3+a_4\tau,a_5+a_6\tau),
\end{align*}
we get $\tau\gamma(\varepsilon(\alpha))=\gamma(\varepsilon(\beta))$.
Let $\PhH$ be the set of positive roots of the root system of type $H_3$. With respect to any simple system $\Delta$, the coordinates of the elements of $\Phi^+$ are in $\ZZ[\tau]$.
Now $\varepsilon^{-1}(\gamma^{-1}(\PhH))$ is just the left column of \fref{fig:posroD6}.

In other words, there is a subspace $U\le \RR^6$ such that the restriction of the reflection arrangement of type $D_6$ to $U$ is the Coxeter arrangement of type $H_3$. Under this restriction, pairs of roots in $\PhD$ are mapped to pairs $\alpha,\tau\alpha\in\RR^3$ where $\alpha$ is a positive root of type $H_3$.
For each such pair, choose the lexicographically greater element in $\ZZ^6$ with respect to $v_1,\ldots,v_6$. This way, for each element of $\PhH$ we get a choice $1$ or $\tau$, i.e.\ a map $\sigma : \PhH \rightarrow \{1,\tau\}$.
The poset $(\{\sigma(\alpha)\alpha\mid \alpha\in \PhH\},\leq)$, where $\beta \leq \beta'$ is given by $\beta' - \beta \in \NN^3$, is (isomorphic to) the poset in \fref{fig_posetH3}(right).

\section{Posets of type \texorpdfstring{$H_4$}{H4}}
\label{sec:H4}

Let $W$ be the Coxeter group of type $H_4$, i.e.\ the group with Coxeter graph:
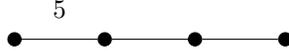
\begin{center}
\begin{tikzpicture}[baseline=4mm,x=12mm,y=4mm]
  \node[edot](1) at (1,2) {};
  \node[edot](2) at (2,2) {};
  \node[edot](3) at (3,2) {};
  \node[edot](4) at (4,2) {};
  \node(6) at (1.5,3) {5};
  \draw (1)--(2)(2)--(3)(3)--(4);
\end{tikzpicture}
\end{center}

Throughout this section, assume that
\[ (\Gamma=\{1,\ldots,60\},\prec)\]
is a poset satisfying Properties~\ref{prop:parabolicsubsystems} through~\ref{prop:antichainsize}. This is to say that it has $4$ minimal elements, a unique maximal element, and behaves well with respect to standard parabolic subsystems. Moreover, its grading is given as shown in \fref{rfcovers}, compare \tref{tab:ranks}.
Its number $n_{k,m}$ of antichains $A\subseteq \Gamma$ of $\Gamma$ with $m$ elements such that $|A\cap \{1,2,3,4\}|=k$ is shown by Figure~\ref{antinum}, where the entry in row $k$ and column $m$ is $n_{k,m}$, compare \tref{tab:triangle}.
\begin{figure}
\begin{tabular}{l | r r r r r}
& 0 & 1 & 2 & 3 & 4 \\ \hline
0& 1 & 56 & 133 & 42 &  0 \\
1& 0 &  4 & 19 &  9 &  0 \\
2& 0 &  0 &  6 &  5 &  0 \\
3& 0 &  0 &  0 &  4 &  0 \\
4& 0 &  0 &  0 &  0 &  1
\end{tabular}
\caption{Expected number of antichains depending on the number of simple roots they contain for root posets of type $H_4$.\label{antinum}}
\end{figure}

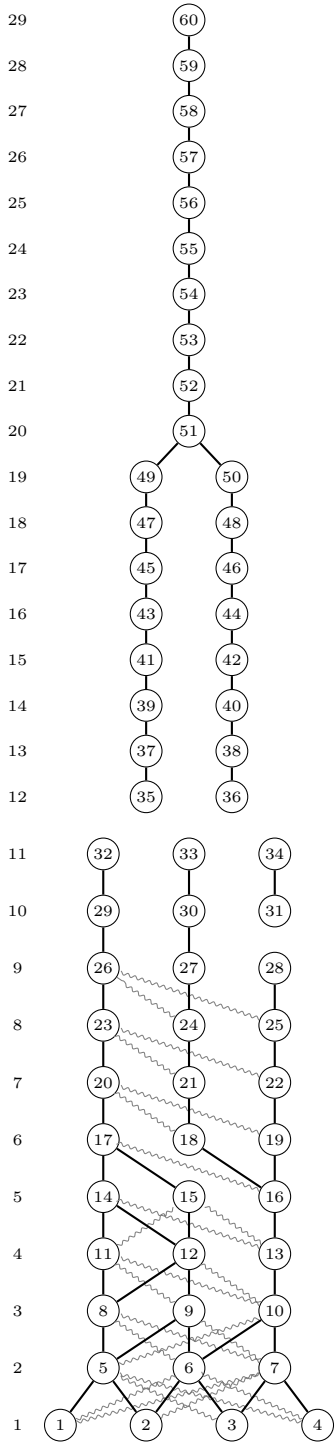
\begin{figure}
{\tiny
\begin{tikzpicture}[scale=.95,baseline=15mm,x=6mm,y=8mm]
\node(l1) at (0,1) {1};
\node(l2) at (0,2) {2};
\node(l3) at (0,3) {3};
\node(l4) at (0,4) {4};
\node(l5) at (0,5) {5};
\node(l6) at (0,6) {6};
\node(l7) at (0,7) {7};
\node(l8) at (0,8) {8};
\node(l9) at (0,9) {9};
\node(l10) at (0,10) {10};
\node(l11) at (0,11) {11};
\node(l12) at (0,12) {12};
\node(l13) at (0,12.8) {13};
\node(l14) at (0,13.6) {14};
\node(l15) at (0,14.4) {15};
\node(l16) at (0,15.2) {16};
\node(l17) at (0,16) {17};
\node(l18) at (0,16.8) {18};
\node(l19) at (0,17.6) {19};
\node(l20) at (0,18.4) {20};
\node(l21) at (0,19.2) {21};
\node(l22) at (0,20) {22};
\node(l23) at (0,20.8) {23};
\node(l24) at (0,21.6) {24};
\node(l25) at (0,22.4) {25};
\node(l26) at (0,23.2) {26};
\node(l27) at (0,24) {27};
\node(l28) at (0,24.8) {28};
\node(l29) at (0,25.6) {29};

\node[ndot](1) at (1,1) {1};
\node[ndot](2) at (3,1) {2};
\node[ndot](3) at (5,1) {3};
\node[ndot](4) at (7,1) {4};

\node[ndot](5) at (2,2) {5};
\node[ndot](6) at (4,2) {6};
\node[ndot](7) at (6,2) {7};
\node[ndot](8) at (2,3) {8};
\node[ndot](9) at (4,3) {9};
\node[ndot](10) at (6,3) {10};
\node[ndot](11) at (2,4) {11};
\node[ndot](12) at (4,4) {12};
\node[ndot](13) at (6,4) {13};
\node[ndot](14) at (2,5) {14};
\node[ndot](15) at (4,5) {15};
\node[ndot](16) at (6,5) {16};
\node[ndot](17) at (2,6) {17};
\node[ndot](18) at (4,6) {18};
\node[ndot](19) at (6,6) {19};
\node[ndot](20) at (2,7) {20};
\node[ndot](21) at (4,7) {21};
\node[ndot](22) at (6,7) {22};
\node[ndot](23) at (2,8) {23};
\node[ndot](24) at (4,8) {24};
\node[ndot](25) at (6,8) {25};
\node[ndot](26) at (2,9) {26};
\node[ndot](27) at (4,9) {27};
\node[ndot](28) at (6,9) {28};
\node[ndot](29) at (2,10) {29};
\node[ndot](30) at (4,10) {30};
\node[ndot](31) at (6,10) {31};
\node[ndot](32) at (2,11) {32};
\node[ndot](33) at (4,11) {33};
\node[ndot](34) at (6,11) {34};

\node[ndot](35) at (3,12) {35};
\node[ndot](36) at (5,12) {36};
\node[ndot](37) at (3,12.8) {37};
\node[ndot](38) at (5,12.8) {38};
\node[ndot](39) at (3,13.6) {39};
\node[ndot](40) at (5,13.6) {40};
\node[ndot](41) at (3,14.4) {41};
\node[ndot](42) at (5,14.4) {42};
\node[ndot](43) at (3,15.2) {43};
\node[ndot](44) at (5,15.2) {44};
\node[ndot](45) at (3,16) {45};
\node[ndot](46) at (5,16) {46};
\node[ndot](47) at (3,16.8) {47};
\node[ndot](48) at (5,16.8) {48};
\node[ndot](49) at (3,17.6) {49};
\node[ndot](50) at (5,17.6) {50};

\node[ndot](51) at (4,18.4) {51};
\node[ndot](52) at (4,19.2) {52};
\node[ndot](53) at (4,20) {53};
\node[ndot](54) at (4,20.8) {54};
\node[ndot](55) at (4,21.6) {55};
\node[ndot](56) at (4,22.4) {56};
\node[ndot](57) at (4,23.2) {57};
\node[ndot](58) at (4,24) {58};
\node[ndot](59) at (4,24.8) {59};
\node[ndot](60) at (4,25.6) {60};

\begin{scope}[decoration={snake,amplitude=.2mm,
        segment length=1mm,post length=0.3mm}]
\draw [decorate,gray] (4)--(6) (4)--(5) (3)--(5) (2)--(7) (1)--(7) (1)--(6) (7)--(9) (7)--(8) (6)--(8) (5)--(10) (10)--(12) (10)--(11) (9)--(11) (13)--(15) (13)--(14) (11)--(15) (16)--(17)
(19)--(20) (18)--(20) (22)--(23) (21)--(23) (25)--(26) (24)--(26);
\end{scope}

\begin{scope}[line width=0.8pt]
\draw (60)--(59) (59)--(58) (58)--(57) (57)--(56) (56)--(55) (55)--(54) (54)--(53) (53)--(52) (52)--(51) (51)--(49) (51)--(50) (49)--(47) (50)--(48) (47)--(45) (48)--(46) (45)--(43)
(46)--(44) (43)--(41) (44)--(42) (41)--(39) (42)--(40) (39)--(37) (40)--(38) (37)--(35) (38)--(36) (32)--(29) (33)--(30) (34)--(31) (29)--(26) (30)--(27) (26)--(23) (27)--(24)
(28)--(25) (23)--(20) (24)--(21) (25)--(22) (20)--(17) (21)--(18) (22)--(19) (17)--(14) (17)--(15) (18)--(16) (19)--(16) (14)--(11) (14)--(12) (15)--(12) (16)--(13) (11)--(8)
(12)--(8) (12)--(9) (13)--(10) (8)--(5) (9)--(5) (9)--(6) (10)--(6) (10)--(7) (5)--(1) (5)--(2) (6)--(2) (6)--(3) (7)--(3) (7)--(4);
\end{scope}

\end{tikzpicture}
}
\caption{Required and forbidden covers in root posets of type $H_4$ satisfying Properties~\ref{prop:parabolicsubsystems} through~\ref{prop:antichainsize}.\label{rfcovers}}
\end{figure}

\subsection{Some consequences of the assumptions}

\begin{propo}
We may assume without loss of generality that $\Gamma$ is build upon Figure~\ref{rfcovers}. I.e., $\Gamma$ contains at least the continuously drawn covers and does not contain the waved ones.
\end{propo}

\begin{proof}
Most of the covers in Figure~\ref{rfcovers} follow the assumption that $\Gamma$ behaves well with respect to standard parabolic subsystems. Here, we assume the ordering of the simple roots in $\Gamma$ to be the same as shown in the above Coxeter graph.

The following covers have another explanation.
The covers in ranks $19$ to $29$ are necessary to ensure the uniqueness of the minimal and maximal elements. For ranks $12$ to $19$, we can assume for the same reason that, up to reodering of the vertices in the upper rank of two consecutive ranks, that we have at least the following covers for each level:
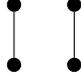
\begin{center}
\begin{tikzpicture}[baseline=8mm,x=8mm,y=8mm]
  \node[edot](1) at (1,1) {};
  \node[edot](2) at (2,1) {};
  \node[edot](3) at (1,2) {};
  \node[edot](4) at (2,2) {};
  \draw(1)--(3)(2)--(4);
\end{tikzpicture}
\end{center}
Again for the same reason, we have the covers $10 \prec 13 \prec 16 \prec 18 $, and $16 \prec 19$, and as well $18 \prec 21 \prec 24 \prec 27$ and $19 \prec 22 \prec 25 \prec 28$.

There has to be a cover for $26$ and one for $27$. If these are equal, say $29$, and if there is no further such cover, $26,27\not\prec 30,31$, then $\{26,27,30,31\}$ is an antichain with $4$ elements, contradicting the assumption that the set of minimal elements is the unique antichain of size $4$, compare \fref{antinum}.
Hence up to permutation of $29,30,31$ we may assume $26 \prec 29$ and $27 \prec 30$.

Now we come to ranks $10$ and $11$: we have a cover for each of $29,30,31$ and each of $32,33,34$ is a cover. Up to permutation of $32,33,34$, we have at least
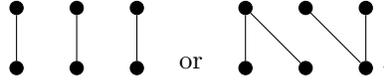
\begin{center}
\begin{tikzpicture}[baseline=8mm,x=8mm,y=8mm]
  \node[edot](1) at (1,1) {};
  \node[edot](2) at (2,1) {};
  \node[edot](3) at (3,1) {};
  \node[edot](4) at (1,2) {};
  \node[edot](5) at (2,2) {};
  \node[edot](6) at (3,2) {};
  \draw(1)--(4)(2)--(5)(3)--(6);
\end{tikzpicture}
\quad or \quad
\begin{tikzpicture}[baseline=8mm,x=8mm,y=8mm]
  \node[edot](1) at (1,1) {};
  \node[edot](2) at (2,1) {};
  \node[edot](3) at (3,1) {};
  \node[edot](4) at (1,2) {};
  \node[edot](5) at (2,2) {};
  \node[edot](6) at (3,2) {};
  \draw(1)--(4)(2)--(4)(3)--(5)(3)--(6);
\end{tikzpicture}
.
\end{center}
But to avoid again an antichain with $4$ elements in the second case, we need to add a cover. Adding one further cover always yields a situation which may be permuted to include the first case, except for
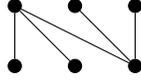
\begin{center}
\begin{tikzpicture}[baseline=8mm,x=8mm,y=8mm]
  \node[edot](1) at (1,1) {};
  \node[edot](2) at (2,1) {};
  \node[edot](3) at (3,1) {};
  \node[edot](4) at (1,2) {};
  \node[edot](5) at (2,2) {};
  \node[edot](6) at (3,2) {};
  \draw(1)--(4)(2)--(4)(3)--(4)(3)--(5)(3)--(6);
\end{tikzpicture}
\end{center}
which again has an antichain with $4$ elements. Thus we may assume the covers $29 \prec 32$, $30 \prec 33$, $31 \prec 34$.
\end{proof}

\subsection{Assuming as well Property~\ref{prop:antichainorbitsize}}
\label{sec:H4algorithms}

In this section, we moreover assume that $\Gamma$ satisfies Property~\ref{prop:antichainorbitsize}.
First, observe that we can predict two further relations in \fref{rfcovers}, namely the covers $9 \prec 13$ and the noncover $8 \not\prec 13$. Assume that $8 \prec 13$. Then under the Panyushev action we have the orbit
\[
  \{8\}, \{3,4,11\}, \{6,7\}, \{1,10\}, \{5\}, \{3,4,8\}, \{6,7,11\}, \{9,10\},
\]
independently of the choice whether we have the cover $9 \prec 13$ or not. Since it has length $8$, this would contradict Property~\ref{prop:antichainorbitsize}. On the other hand, removing the vertex $1$ should result in a poset of type $A_3$, which is now only possible if $9 \prec 13$.

\medskip

We present two algorithms to prove the nonexistence of a poset satisfying Properties~\ref{prop:parabolicsubsystems} through~\ref{prop:antichainsize} and also Property~\ref{prop:antichainorbitsize}.
Both variants need functions to count or to compute antichains of a given poset, and a function to compute the lengths of the orbits under the Panyushev action.

\begin{remar}
We did not use computer algebra systems or similar scripting languages for several reasons.
Our computations do not require higher functions or libraries, and the occurring arithmetic is always with small numbers. Further, since our posets have at most $60$ elements, we can for instance store the information of an antichain or of an order ideal in a very compact $64$-bit variable. Last but not least, a compiled and optimized version always performs much faster than interpreted code. This can be decisive if the runtime is in the range of a few hours to a few days.
\end{remar}

\subsubsection{The first algorithm}

There are $37$ potential covers $i\prec j$ with $i,j\le 36$ that may be added to Figure~\ref{rfcovers} together with the cover $9 \prec 13$ and the noncover $8 \prec 13$. Let us call them $c_1,\ldots,c_{37}$. Then there are $14$ further potential covers in the upper part, say $w_1,\ldots,w_{14}$.
We divide the algorithm into two parts.

\subsubsection*{The upper part} First we compute all $2^{14}$ possible posets on the vertices $35,\ldots,60$ and store those which do not contain an entire orbit under the Panyushev action of length other than $2$,$3$,$5$, or $30$ in a list $R$. For each element $\Gamma_0$ of $R$, we also compute the number $a_2(\Gamma_0)$ of antichains with two elements, the number $g_{35}(\Gamma_0)$ of elements not greater or equal than $35$, and the number $g_{36}(\Gamma_0)$ of elements not greater or equal than $36$.

\subsubsection*{The lower part} For the covers between vertices in $1,\ldots,36$, we use the following recursion.

\algo{EnumeratePosets1}{$\Gamma'$,$p$}{Enumerates all root posets for $H_4$ satisfying Properties~\ref{prop:parabolicsubsystems} through~\ref{prop:antichainorbitsize}, starting from~$\Gamma'$}
{A poset $\Gamma'$, an index $p$ for the next potential cover.}
{Posets based upon $\Gamma'$ satisfying Properties~\ref{prop:parabolicsubsystems} through~\ref{prop:antichainorbitsize}.}
{
\item If $p=38$ then:
\begin{enumerate}
\item Count the antichains with $3$ elements within $1,\ldots,36$. If this number is greater than $60$, then return $0$.
\item Compute all orbits under the Panyushev action which are entirely contained in $1,\ldots,36$. As soon as one of them has length other than $2$,$3$,$5$, or $30$, return $0$.
\item Within $5,\ldots,36$: Count the number $b_2$ of antichains with $2$ elements, the number $l_{35}$ of elements not less or equal to $35$, and the number $l_{36}$ of elements not less or equal to $36$.
\item For each $\Gamma_0\in R$ with $a_2(\Gamma_0)+b_2+l_{35}\cdot g_{36}(\Gamma_0)+l_{36}\cdot g_{35}(\Gamma_0) = 133$, compute the poset $\Gamma$ as a combination of $\Gamma'$ and $\Gamma_0$. If $\Gamma$ satisfies all $5$ properties, then print $\Gamma$ to an output file.
\end{enumerate}
\item If all chosen covers between degree $4$ and $5$ have been included, then check that $\{14,15,18,19\}$ is not an antichain with $4$ elements and return $0$ if it is.
\item If all chosen covers between degree $8$ and $9$ have been included, then check that $27\prec 29$ or $28\prec 29$ and that $27\prec 31$ or $28\prec 31$ and return $0$ if not. Further check that $\{26,\ldots,31\}$ does not contain an antichain with $4$ elements.
\item If all chosen covers between degree $10$ and $11$ have been included, then check that each of $35,36$ is a cover and that each of $32,33,34$ is covered. Otherwise return $0$.\label{step_23}
\item Call {\tt EnumeratePosets1}($\Gamma'$,$p+1$).
\item Include the cover $c_p$ to $\Gamma'$ to a new poset $\Gamma''$.
\item Call {\tt EnumeratePosets1}($\Gamma''$,$p+1$).
}
We use the data of Figure~\ref{rfcovers} to compute an initial poset $\Gamma'$ for the first call {\tt EnumeratePosets1}($\Gamma'$,$1$).

\begin{remar}
It is easy to parallelize Algorithm~\ref{EnumeratePosets1}.
For instance, call $c_1,\ldots,c_6$ the potential covers within $32,\ldots,36$.
Now one can fork (duplicate the program) for each of the $25$ possible choices for $c_1,\ldots,c_6$ (some of the $2^6$ choices are excluded by Step~\ref{step_23}).
\end{remar}

\subsubsection{Results}

Algorithm~\ref{EnumeratePosets1} terminates after approximately two days of CPU-time and does not find any poset satisfying all the above conditions.
However, we slightly modified the algorithm, and found over 80 billion posets only satisfying Properties~\ref{prop:parabolicsubsystems} through~\ref{prop:antichainsize}. We guess that there are several 100 billions of them.
We show four examples of such posets in \fref{fig:rp12345_2}.
Property~\ref{prop:antichainorbitsize} is much more restrictive and turns out to be the best break condition for the algorithm.

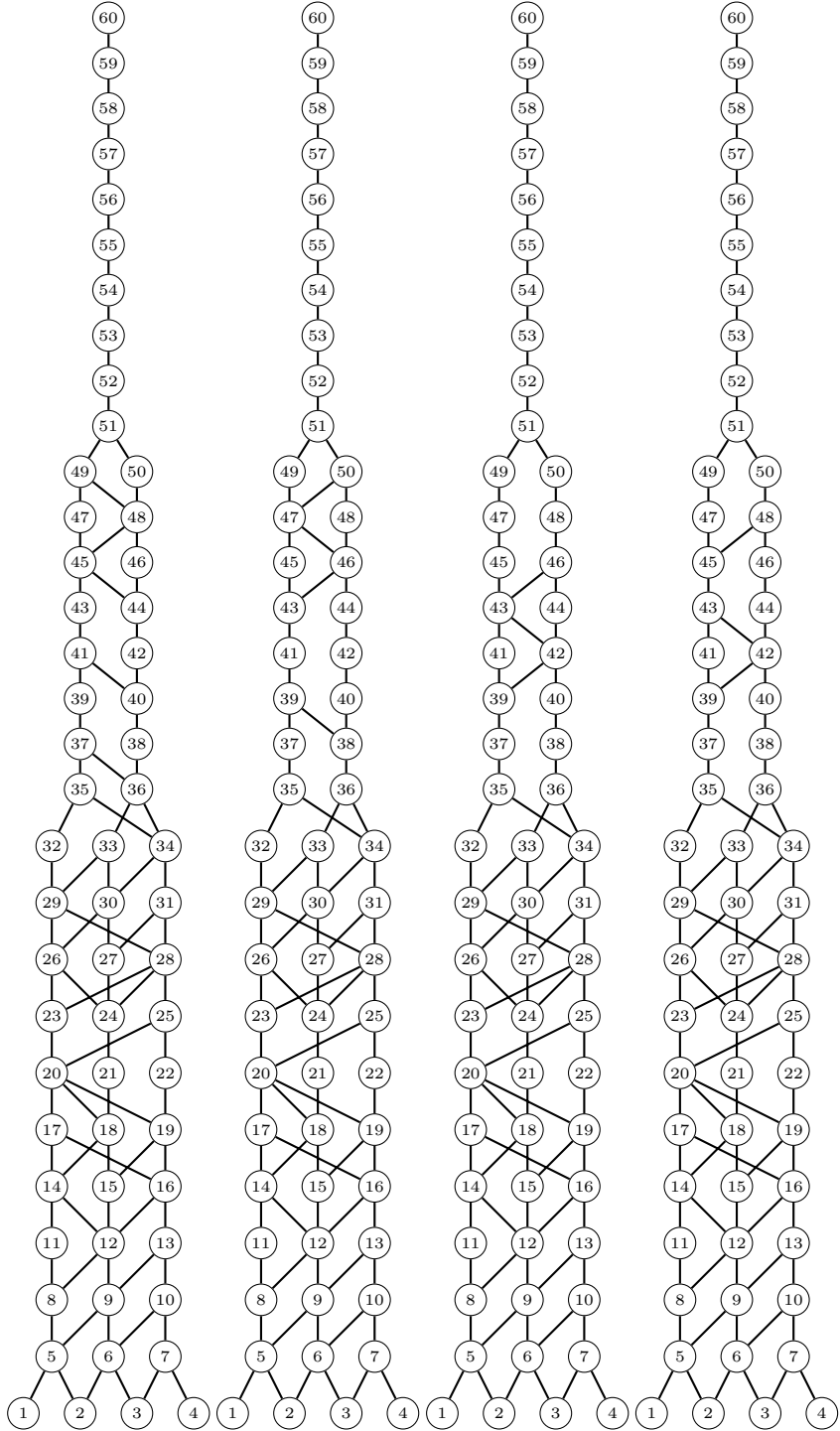
\begin{figure}
{\tiny
\begin{tikzpicture}[scale=.95,baseline=15mm,x=4mm,y=8mm]
\node[ndot](1) at (1,1) {1};
\node[ndot](2) at (3,1) {2};
\node[ndot](3) at (5,1) {3};
\node[ndot](4) at (7,1) {4};

\node[ndot](5) at (2,2) {5};
\node[ndot](6) at (4,2) {6};
\node[ndot](7) at (6,2) {7};
\node[ndot](8) at (2,3) {8};
\node[ndot](9) at (4,3) {9};
\node[ndot](10) at (6,3) {10};
\node[ndot](11) at (2,4) {11};
\node[ndot](12) at (4,4) {12};
\node[ndot](13) at (6,4) {13};
\node[ndot](14) at (2,5) {14};
\node[ndot](15) at (4,5) {15};
\node[ndot](16) at (6,5) {16};
\node[ndot](17) at (2,6) {17};
\node[ndot](18) at (4,6) {18};
\node[ndot](19) at (6,6) {19};
\node[ndot](20) at (2,7) {20};
\node[ndot](21) at (4,7) {21};
\node[ndot](22) at (6,7) {22};
\node[ndot](23) at (2,8) {23};
\node[ndot](24) at (4,8) {24};
\node[ndot](25) at (6,8) {25};
\node[ndot](26) at (2,9) {26};
\node[ndot](27) at (4,9) {27};
\node[ndot](28) at (6,9) {28};
\node[ndot](29) at (2,10) {29};
\node[ndot](30) at (4,10) {30};
\node[ndot](31) at (6,10) {31};
\node[ndot](32) at (2,11) {32};
\node[ndot](33) at (4,11) {33};
\node[ndot](34) at (6,11) {34};

\node[ndot](35) at (3,12) {35};
\node[ndot](36) at (5,12) {36};
\node[ndot](37) at (3,12.8) {37};
\node[ndot](38) at (5,12.8) {38};
\node[ndot](39) at (3,13.6) {39};
\node[ndot](40) at (5,13.6) {40};
\node[ndot](41) at (3,14.4) {41};
\node[ndot](42) at (5,14.4) {42};
\node[ndot](43) at (3,15.2) {43};
\node[ndot](44) at (5,15.2) {44};
\node[ndot](45) at (3,16) {45};
\node[ndot](46) at (5,16) {46};
\node[ndot](47) at (3,16.8) {47};
\node[ndot](48) at (5,16.8) {48};
\node[ndot](49) at (3,17.6) {49};
\node[ndot](50) at (5,17.6) {50};

\node[ndot](51) at (4,18.4) {51};
\node[ndot](52) at (4,19.2) {52};
\node[ndot](53) at (4,20) {53};
\node[ndot](54) at (4,20.8) {54};
\node[ndot](55) at (4,21.6) {55};
\node[ndot](56) at (4,22.4) {56};
\node[ndot](57) at (4,23.2) {57};
\node[ndot](58) at (4,24) {58};
\node[ndot](59) at (4,24.8) {59};
\node[ndot](60) at (4,25.6) {60};

\begin{scope}[line width=0.8pt]
\draw (60)--(59)(59)--(58)(58)--(57)(57)--(56)(56)--(55)(55)--(54)(54)--(53)(53)--(52)(52)--(51)(51)--(50)(51)--(49)(50)--(48)(49)--(48)(49)--(47)(48)--(46)(48)--(45)(47)--(45)(46)--(44)(45)--(44)(45)--(43)(44)--(42)(43)--(41)(42)--(40)(41)--(40)(41)--(39)(40)--(38)(39)--(37)(38)--(36)(37)--(36)(37)--(35)(36)--(34)(36)--(33)(35)--(34)(35)--(32)(34)--(31)(34)--(30)(33)--(30)(33)--(29)(32)--(29)(31)--(28)(31)--(27)(30)--(27)(30)--(26)(29)--(28)(29)--(26)(28)--(25)(28)--(24)(28)--(23)(27)--(24)(26)--(24)(26)--(23)(25)--(22)(25)--(20)(24)--(21)(23)--(20)(22)--(19)(21)--(18)(20)--(19)(20)--(18)(20)--(17)(19)--(16)(19)--(15)(18)--(15)(18)--(14)(17)--(16)(17)--(14)(16)--(13)(16)--(12)(15)--(12)(14)--(12)(14)--(11)(13)--(10)(13)--(9)(12)--(9)(12)--(8)(11)--(8)(10)--(7)(10)--(6)(9)--(6)(9)--(5)(8)--(5)(7)--(4)(7)--(3)(6)--(3)(6)--(2)(5)--(2)(5)--(1);
\end{scope}

\end{tikzpicture}
\begin{tikzpicture}[scale=.95,baseline=15mm,x=4mm,y=8mm]
\node[ndot](1) at (1,1) {1};
\node[ndot](2) at (3,1) {2};
\node[ndot](3) at (5,1) {3};
\node[ndot](4) at (7,1) {4};

\node[ndot](5) at (2,2) {5};
\node[ndot](6) at (4,2) {6};
\node[ndot](7) at (6,2) {7};
\node[ndot](8) at (2,3) {8};
\node[ndot](9) at (4,3) {9};
\node[ndot](10) at (6,3) {10};
\node[ndot](11) at (2,4) {11};
\node[ndot](12) at (4,4) {12};
\node[ndot](13) at (6,4) {13};
\node[ndot](14) at (2,5) {14};
\node[ndot](15) at (4,5) {15};
\node[ndot](16) at (6,5) {16};
\node[ndot](17) at (2,6) {17};
\node[ndot](18) at (4,6) {18};
\node[ndot](19) at (6,6) {19};
\node[ndot](20) at (2,7) {20};
\node[ndot](21) at (4,7) {21};
\node[ndot](22) at (6,7) {22};
\node[ndot](23) at (2,8) {23};
\node[ndot](24) at (4,8) {24};
\node[ndot](25) at (6,8) {25};
\node[ndot](26) at (2,9) {26};
\node[ndot](27) at (4,9) {27};
\node[ndot](28) at (6,9) {28};
\node[ndot](29) at (2,10) {29};
\node[ndot](30) at (4,10) {30};
\node[ndot](31) at (6,10) {31};
\node[ndot](32) at (2,11) {32};
\node[ndot](33) at (4,11) {33};
\node[ndot](34) at (6,11) {34};

\node[ndot](35) at (3,12) {35};
\node[ndot](36) at (5,12) {36};
\node[ndot](37) at (3,12.8) {37};
\node[ndot](38) at (5,12.8) {38};
\node[ndot](39) at (3,13.6) {39};
\node[ndot](40) at (5,13.6) {40};
\node[ndot](41) at (3,14.4) {41};
\node[ndot](42) at (5,14.4) {42};
\node[ndot](43) at (3,15.2) {43};
\node[ndot](44) at (5,15.2) {44};
\node[ndot](45) at (3,16) {45};
\node[ndot](46) at (5,16) {46};
\node[ndot](47) at (3,16.8) {47};
\node[ndot](48) at (5,16.8) {48};
\node[ndot](49) at (3,17.6) {49};
\node[ndot](50) at (5,17.6) {50};

\node[ndot](51) at (4,18.4) {51};
\node[ndot](52) at (4,19.2) {52};
\node[ndot](53) at (4,20) {53};
\node[ndot](54) at (4,20.8) {54};
\node[ndot](55) at (4,21.6) {55};
\node[ndot](56) at (4,22.4) {56};
\node[ndot](57) at (4,23.2) {57};
\node[ndot](58) at (4,24) {58};
\node[ndot](59) at (4,24.8) {59};
\node[ndot](60) at (4,25.6) {60};

\begin{scope}[line width=0.8pt]
\draw (60)--(59)(59)--(58)(58)--(57)(57)--(56)(56)--(55)(55)--(54)(54)--(53)(53)--(52)(52)--(51)(51)--(50)(51)--(49)(50)--(48)(50)--(47)(49)--(47)(48)--(46)(47)--(46)(47)--(45)(46)--(44)(46)--(43)(45)--(43)(44)--(42)(43)--(41)(42)--(40)(41)--(39)(40)--(38)(39)--(38)(39)--(37)(38)--(36)(37)--(35)(36)--(34)(36)--(33)(35)--(34)(35)--(32)(34)--(31)(34)--(30)(33)--(30)(33)--(29)(32)--(29)(31)--(28)(31)--(27)(30)--(27)(30)--(26)(29)--(28)(29)--(26)(28)--(25)(28)--(24)(28)--(23)(27)--(24)(26)--(24)(26)--(23)(25)--(22)(25)--(20)(24)--(21)(23)--(20)(22)--(19)(21)--(18)(20)--(19)(20)--(18)(20)--(17)(19)--(16)(19)--(15)(18)--(15)(18)--(14)(17)--(16)(17)--(14)(16)--(13)(16)--(12)(15)--(12)(14)--(12)(14)--(11)(13)--(10)(13)--(9)(12)--(9)(12)--(8)(11)--(8)(10)--(7)(10)--(6)(9)--(6)(9)--(5)(8)--(5)(7)--(4)(7)--(3)(6)--(3)(6)--(2)(5)--(2)(5)--(1);
\end{scope}

\end{tikzpicture}
\begin{tikzpicture}[scale=.95,baseline=15mm,x=4mm,y=8mm]
\node[ndot](1) at (1,1) {1};
\node[ndot](2) at (3,1) {2};
\node[ndot](3) at (5,1) {3};
\node[ndot](4) at (7,1) {4};

\node[ndot](5) at (2,2) {5};
\node[ndot](6) at (4,2) {6};
\node[ndot](7) at (6,2) {7};
\node[ndot](8) at (2,3) {8};
\node[ndot](9) at (4,3) {9};
\node[ndot](10) at (6,3) {10};
\node[ndot](11) at (2,4) {11};
\node[ndot](12) at (4,4) {12};
\node[ndot](13) at (6,4) {13};
\node[ndot](14) at (2,5) {14};
\node[ndot](15) at (4,5) {15};
\node[ndot](16) at (6,5) {16};
\node[ndot](17) at (2,6) {17};
\node[ndot](18) at (4,6) {18};
\node[ndot](19) at (6,6) {19};
\node[ndot](20) at (2,7) {20};
\node[ndot](21) at (4,7) {21};
\node[ndot](22) at (6,7) {22};
\node[ndot](23) at (2,8) {23};
\node[ndot](24) at (4,8) {24};
\node[ndot](25) at (6,8) {25};
\node[ndot](26) at (2,9) {26};
\node[ndot](27) at (4,9) {27};
\node[ndot](28) at (6,9) {28};
\node[ndot](29) at (2,10) {29};
\node[ndot](30) at (4,10) {30};
\node[ndot](31) at (6,10) {31};
\node[ndot](32) at (2,11) {32};
\node[ndot](33) at (4,11) {33};
\node[ndot](34) at (6,11) {34};

\node[ndot](35) at (3,12) {35};
\node[ndot](36) at (5,12) {36};
\node[ndot](37) at (3,12.8) {37};
\node[ndot](38) at (5,12.8) {38};
\node[ndot](39) at (3,13.6) {39};
\node[ndot](40) at (5,13.6) {40};
\node[ndot](41) at (3,14.4) {41};
\node[ndot](42) at (5,14.4) {42};
\node[ndot](43) at (3,15.2) {43};
\node[ndot](44) at (5,15.2) {44};
\node[ndot](45) at (3,16) {45};
\node[ndot](46) at (5,16) {46};
\node[ndot](47) at (3,16.8) {47};
\node[ndot](48) at (5,16.8) {48};
\node[ndot](49) at (3,17.6) {49};
\node[ndot](50) at (5,17.6) {50};

\node[ndot](51) at (4,18.4) {51};
\node[ndot](52) at (4,19.2) {52};
\node[ndot](53) at (4,20) {53};
\node[ndot](54) at (4,20.8) {54};
\node[ndot](55) at (4,21.6) {55};
\node[ndot](56) at (4,22.4) {56};
\node[ndot](57) at (4,23.2) {57};
\node[ndot](58) at (4,24) {58};
\node[ndot](59) at (4,24.8) {59};
\node[ndot](60) at (4,25.6) {60};

\begin{scope}[line width=0.8pt]
\draw (60)--(59)(59)--(58)(58)--(57)(57)--(56)(56)--(55)(55)--(54)(54)--(53)(53)--(52)(52)--(51)(51)--(50)(51)--(49)(50)--(48)(49)--(47)(48)--(46)(47)--(45)(46)--(44)(46)--(43)(45)--(43)(44)--(42)(43)--(42)(43)--(41)(42)--(40)(42)--(39)(41)--(39)(40)--(38)(39)--(37)(38)--(36)(37)--(35)(36)--(34)(36)--(33)(35)--(34)(35)--(32)(34)--(31)(34)--(30)(33)--(30)(33)--(29)(32)--(29)(31)--(28)(31)--(27)(30)--(27)(30)--(26)(29)--(28)(29)--(26)(28)--(25)(28)--(24)(28)--(23)(27)--(24)(26)--(24)(26)--(23)(25)--(22)(25)--(20)(24)--(21)(23)--(20)(22)--(19)(21)--(18)(20)--(19)(20)--(18)(20)--(17)(19)--(16)(19)--(15)(18)--(15)(18)--(14)(17)--(16)(17)--(14)(16)--(13)(16)--(12)(15)--(12)(14)--(12)(14)--(11)(13)--(10)(13)--(9)(12)--(9)(12)--(8)(11)--(8)(10)--(7)(10)--(6)(9)--(6)(9)--(5)(8)--(5)(7)--(4)(7)--(3)(6)--(3)(6)--(2)(5)--(2)(5)--(1);
\end{scope}

\end{tikzpicture}
\begin{tikzpicture}[scale=.95,baseline=15mm,x=4mm,y=8mm]
\node[ndot](1) at (1,1) {1};
\node[ndot](2) at (3,1) {2};
\node[ndot](3) at (5,1) {3};
\node[ndot](4) at (7,1) {4};

\node[ndot](5) at (2,2) {5};
\node[ndot](6) at (4,2) {6};
\node[ndot](7) at (6,2) {7};
\node[ndot](8) at (2,3) {8};
\node[ndot](9) at (4,3) {9};
\node[ndot](10) at (6,3) {10};
\node[ndot](11) at (2,4) {11};
\node[ndot](12) at (4,4) {12};
\node[ndot](13) at (6,4) {13};
\node[ndot](14) at (2,5) {14};
\node[ndot](15) at (4,5) {15};
\node[ndot](16) at (6,5) {16};
\node[ndot](17) at (2,6) {17};
\node[ndot](18) at (4,6) {18};
\node[ndot](19) at (6,6) {19};
\node[ndot](20) at (2,7) {20};
\node[ndot](21) at (4,7) {21};
\node[ndot](22) at (6,7) {22};
\node[ndot](23) at (2,8) {23};
\node[ndot](24) at (4,8) {24};
\node[ndot](25) at (6,8) {25};
\node[ndot](26) at (2,9) {26};
\node[ndot](27) at (4,9) {27};
\node[ndot](28) at (6,9) {28};
\node[ndot](29) at (2,10) {29};
\node[ndot](30) at (4,10) {30};
\node[ndot](31) at (6,10) {31};
\node[ndot](32) at (2,11) {32};
\node[ndot](33) at (4,11) {33};
\node[ndot](34) at (6,11) {34};

\node[ndot](35) at (3,12) {35};
\node[ndot](36) at (5,12) {36};
\node[ndot](37) at (3,12.8) {37};
\node[ndot](38) at (5,12.8) {38};
\node[ndot](39) at (3,13.6) {39};
\node[ndot](40) at (5,13.6) {40};
\node[ndot](41) at (3,14.4) {41};
\node[ndot](42) at (5,14.4) {42};
\node[ndot](43) at (3,15.2) {43};
\node[ndot](44) at (5,15.2) {44};
\node[ndot](45) at (3,16) {45};
\node[ndot](46) at (5,16) {46};
\node[ndot](47) at (3,16.8) {47};
\node[ndot](48) at (5,16.8) {48};
\node[ndot](49) at (3,17.6) {49};
\node[ndot](50) at (5,17.6) {50};

\node[ndot](51) at (4,18.4) {51};
\node[ndot](52) at (4,19.2) {52};
\node[ndot](53) at (4,20) {53};
\node[ndot](54) at (4,20.8) {54};
\node[ndot](55) at (4,21.6) {55};
\node[ndot](56) at (4,22.4) {56};
\node[ndot](57) at (4,23.2) {57};
\node[ndot](58) at (4,24) {58};
\node[ndot](59) at (4,24.8) {59};
\node[ndot](60) at (4,25.6) {60};

\begin{scope}[line width=0.8pt]
\draw (60)--(59)(59)--(58)(58)--(57)(57)--(56)(56)--(55)(55)--(54)(54)--(53)(53)--(52)(52)--(51)(51)--(50)(51)--(49)(50)--(48)(49)--(47)(48)--(46)(48)--(45)(47)--(45)(46)--(44)(45)--(43)(44)--(42)(43)--(42)(43)--(41)(42)--(40)(42)--(39)(41)--(39)(40)--(38)(39)--(37)(38)--(36)(37)--(35)(36)--(34)(36)--(33)(35)--(34)(35)--(32)(34)--(31)(34)--(30)(33)--(30)(33)--(29)(32)--(29)(31)--(28)(31)--(27)(30)--(27)(30)--(26)(29)--(28)(29)--(26)(28)--(25)(28)--(24)(28)--(23)(27)--(24)(26)--(24)(26)--(23)(25)--(22)(25)--(20)(24)--(21)(23)--(20)(22)--(19)(21)--(18)(20)--(19)(20)--(18)(20)--(17)(19)--(16)(19)--(15)(18)--(15)(18)--(14)(17)--(16)(17)--(14)(16)--(13)(16)--(12)(15)--(12)(14)--(12)(14)--(11)(13)--(10)(13)--(9)(12)--(9)(12)--(8)(11)--(8)(10)--(7)(10)--(6)(9)--(6)(9)--(5)(8)--(5)(7)--(4)(7)--(3)(6)--(3)(6)--(2)(5)--(2)(5)--(1);
\end{scope}

\end{tikzpicture}
}
\caption{All root posets for $H_4$ satisfying the Properties~\ref{prop:parabolicsubsystems} through~\ref{prop:antichainsize}, and Property~\ref{prop:antichainorbitsize}(b).\label{fig:rp12345_2}}
\end{figure}

\subsubsection{The second algorithm}

Since it is very difficult to ensure that an implementation of the above algorithm is free of mistakes, we propose in this section a second slightly different algorithm, and thereby increasing our changes of providing valid results.
Further, this second algorithm is much better in the sense that it takes only a few hours of CPU-time and performs a more general search, namely for posets satisfying all the properties as in the previous sections, except that we do not take the structure of standard parabolic subposets into account.
For each pair of consecutive ranks we get a set of possible configurations of covers.
For example, there are $51$ possible configurations of covers between two ranks containing 3 vertices, up to permutations of the vertices of larger rank:
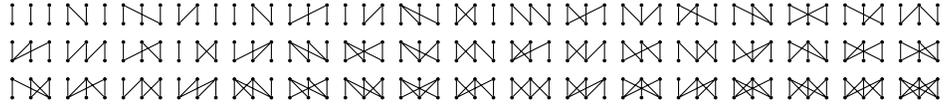
\vspace{10pt}
\begin{center}
\begin{tikzpicture}[baseline=30pt,x=7pt,y=7pt]
\node[cdot](1) at (0,7) {};\node[cdot](2) at (1,7) {};\node[cdot](3) at (2,7) {};\node[cdot](4) at (0,6) {};\node[cdot](5) at (1,6) {};\node[cdot](6) at (2,6) {};
\node[cdot](7) at (3,7) {};\node[cdot](8) at (4,7) {};\node[cdot](9) at (5,7) {};\node[cdot](10) at (3,6) {};\node[cdot](11) at (4,6) {};\node[cdot](12) at (5,6) {};
\node[cdot](13) at (6,7) {};\node[cdot](14) at (7,7) {};\node[cdot](15) at (8,7) {};\node[cdot](16) at (6,6) {};\node[cdot](17) at (7,6) {};\node[cdot](18) at (8,6) {};
\node[cdot](19) at (9,7) {};\node[cdot](20) at (10,7) {};\node[cdot](21) at (11,7) {};\node[cdot](22) at (9,6) {};\node[cdot](23) at (10,6) {};\node[cdot](24) at (11,6) {};
\node[cdot](25) at (12,7) {};\node[cdot](26) at (13,7) {};\node[cdot](27) at (14,7) {};\node[cdot](28) at (12,6) {};\node[cdot](29) at (13,6) {};\node[cdot](30) at (14,6) {};
\node[cdot](31) at (15,7) {};\node[cdot](32) at (16,7) {};\node[cdot](33) at (17,7) {};\node[cdot](34) at (15,6) {};\node[cdot](35) at (16,6) {};\node[cdot](36) at (17,6) {};
\node[cdot](37) at (18,7) {};\node[cdot](38) at (19,7) {};\node[cdot](39) at (20,7) {};\node[cdot](40) at (18,6) {};\node[cdot](41) at (19,6) {};\node[cdot](42) at (20,6) {};
\node[cdot](43) at (21,7) {};\node[cdot](44) at (22,7) {};\node[cdot](45) at (23,7) {};\node[cdot](46) at (21,6) {};\node[cdot](47) at (22,6) {};\node[cdot](48) at (23,6) {};
\node[cdot](49) at (24,7) {};\node[cdot](50) at (25,7) {};\node[cdot](51) at (26,7) {};\node[cdot](52) at (24,6) {};\node[cdot](53) at (25,6) {};\node[cdot](54) at (26,6) {};
\node[cdot](55) at (27,7) {};\node[cdot](56) at (28,7) {};\node[cdot](57) at (29,7) {};\node[cdot](58) at (27,6) {};\node[cdot](59) at (28,6) {};\node[cdot](60) at (29,6) {};
\node[cdot](61) at (30,7) {};\node[cdot](62) at (31,7) {};\node[cdot](63) at (32,7) {};\node[cdot](64) at (30,6) {};\node[cdot](65) at (31,6) {};\node[cdot](66) at (32,6) {};
\node[cdot](67) at (33,7) {};\node[cdot](68) at (34,7) {};\node[cdot](69) at (35,7) {};\node[cdot](70) at (33,6) {};\node[cdot](71) at (34,6) {};\node[cdot](72) at (35,6) {};
\node[cdot](73) at (36,7) {};\node[cdot](74) at (37,7) {};\node[cdot](75) at (38,7) {};\node[cdot](76) at (36,6) {};\node[cdot](77) at (37,6) {};\node[cdot](78) at (38,6) {};
\node[cdot](79) at (39,7) {};\node[cdot](80) at (40,7) {};\node[cdot](81) at (41,7) {};\node[cdot](82) at (39,6) {};\node[cdot](83) at (40,6) {};\node[cdot](84) at (41,6) {};
\node[cdot](85) at (42,7) {};\node[cdot](86) at (43,7) {};\node[cdot](87) at (44,7) {};\node[cdot](88) at (42,6) {};\node[cdot](89) at (43,6) {};\node[cdot](90) at (44,6) {};
\node[cdot](91) at (45,7) {};\node[cdot](92) at (46,7) {};\node[cdot](93) at (47,7) {};\node[cdot](94) at (45,6) {};\node[cdot](95) at (46,6) {};\node[cdot](96) at (47,6) {};
\node[cdot](97) at (48,7) {};\node[cdot](98) at (49,7) {};\node[cdot](99) at (50,7) {};\node[cdot](100) at (48,6) {};\node[cdot](101) at (49,6) {};\node[cdot](102) at (50,6) {};
\node[cdot](103) at (0,5) {};\node[cdot](104) at (1,5) {};\node[cdot](105) at (2,5) {};\node[cdot](106) at (0,4) {};\node[cdot](107) at (1,4) {};\node[cdot](108) at (2,4) {};
\node[cdot](109) at (3,5) {};\node[cdot](110) at (4,5) {};\node[cdot](111) at (5,5) {};\node[cdot](112) at (3,4) {};\node[cdot](113) at (4,4) {};\node[cdot](114) at (5,4) {};
\node[cdot](115) at (6,5) {};\node[cdot](116) at (7,5) {};\node[cdot](117) at (8,5) {};\node[cdot](118) at (6,4) {};\node[cdot](119) at (7,4) {};\node[cdot](120) at (8,4) {};
\node[cdot](121) at (9,5) {};\node[cdot](122) at (10,5) {};\node[cdot](123) at (11,5) {};\node[cdot](124) at (9,4) {};\node[cdot](125) at (10,4) {};\node[cdot](126) at (11,4) {};
\node[cdot](127) at (12,5) {};\node[cdot](128) at (13,5) {};\node[cdot](129) at (14,5) {};\node[cdot](130) at (12,4) {};\node[cdot](131) at (13,4) {};\node[cdot](132) at (14,4) {};
\node[cdot](133) at (15,5) {};\node[cdot](134) at (16,5) {};\node[cdot](135) at (17,5) {};\node[cdot](136) at (15,4) {};\node[cdot](137) at (16,4) {};\node[cdot](138) at (17,4) {};
\node[cdot](139) at (18,5) {};\node[cdot](140) at (19,5) {};\node[cdot](141) at (20,5) {};\node[cdot](142) at (18,4) {};\node[cdot](143) at (19,4) {};\node[cdot](144) at (20,4) {};
\node[cdot](145) at (21,5) {};\node[cdot](146) at (22,5) {};\node[cdot](147) at (23,5) {};\node[cdot](148) at (21,4) {};\node[cdot](149) at (22,4) {};\node[cdot](150) at (23,4) {};
\node[cdot](151) at (24,5) {};\node[cdot](152) at (25,5) {};\node[cdot](153) at (26,5) {};\node[cdot](154) at (24,4) {};\node[cdot](155) at (25,4) {};\node[cdot](156) at (26,4) {};
\node[cdot](157) at (27,5) {};\node[cdot](158) at (28,5) {};\node[cdot](159) at (29,5) {};\node[cdot](160) at (27,4) {};\node[cdot](161) at (28,4) {};\node[cdot](162) at (29,4) {};
\node[cdot](163) at (30,5) {};\node[cdot](164) at (31,5) {};\node[cdot](165) at (32,5) {};\node[cdot](166) at (30,4) {};\node[cdot](167) at (31,4) {};\node[cdot](168) at (32,4) {};
\node[cdot](169) at (33,5) {};\node[cdot](170) at (34,5) {};\node[cdot](171) at (35,5) {};\node[cdot](172) at (33,4) {};\node[cdot](173) at (34,4) {};\node[cdot](174) at (35,4) {};
\node[cdot](175) at (36,5) {};\node[cdot](176) at (37,5) {};\node[cdot](177) at (38,5) {};\node[cdot](178) at (36,4) {};\node[cdot](179) at (37,4) {};\node[cdot](180) at (38,4) {};
\node[cdot](181) at (39,5) {};\node[cdot](182) at (40,5) {};\node[cdot](183) at (41,5) {};\node[cdot](184) at (39,4) {};\node[cdot](185) at (40,4) {};\node[cdot](186) at (41,4) {};
\node[cdot](187) at (42,5) {};\node[cdot](188) at (43,5) {};\node[cdot](189) at (44,5) {};\node[cdot](190) at (42,4) {};\node[cdot](191) at (43,4) {};\node[cdot](192) at (44,4) {};
\node[cdot](193) at (45,5) {};\node[cdot](194) at (46,5) {};\node[cdot](195) at (47,5) {};\node[cdot](196) at (45,4) {};\node[cdot](197) at (46,4) {};\node[cdot](198) at (47,4) {};
\node[cdot](199) at (48,5) {};\node[cdot](200) at (49,5) {};\node[cdot](201) at (50,5) {};\node[cdot](202) at (48,4) {};\node[cdot](203) at (49,4) {};\node[cdot](204) at (50,4) {};
\node[cdot](205) at (0,3) {};\node[cdot](206) at (1,3) {};\node[cdot](207) at (2,3) {};\node[cdot](208) at (0,2) {};\node[cdot](209) at (1,2) {};\node[cdot](210) at (2,2) {};
\node[cdot](211) at (3,3) {};\node[cdot](212) at (4,3) {};\node[cdot](213) at (5,3) {};\node[cdot](214) at (3,2) {};\node[cdot](215) at (4,2) {};\node[cdot](216) at (5,2) {};
\node[cdot](217) at (6,3) {};\node[cdot](218) at (7,3) {};\node[cdot](219) at (8,3) {};\node[cdot](220) at (6,2) {};\node[cdot](221) at (7,2) {};\node[cdot](222) at (8,2) {};
\node[cdot](223) at (9,3) {};\node[cdot](224) at (10,3) {};\node[cdot](225) at (11,3) {};\node[cdot](226) at (9,2) {};\node[cdot](227) at (10,2) {};\node[cdot](228) at (11,2) {};
\node[cdot](229) at (12,3) {};\node[cdot](230) at (13,3) {};\node[cdot](231) at (14,3) {};\node[cdot](232) at (12,2) {};\node[cdot](233) at (13,2) {};\node[cdot](234) at (14,2) {};
\node[cdot](235) at (15,3) {};\node[cdot](236) at (16,3) {};\node[cdot](237) at (17,3) {};\node[cdot](238) at (15,2) {};\node[cdot](239) at (16,2) {};\node[cdot](240) at (17,2) {};
\node[cdot](241) at (18,3) {};\node[cdot](242) at (19,3) {};\node[cdot](243) at (20,3) {};\node[cdot](244) at (18,2) {};\node[cdot](245) at (19,2) {};\node[cdot](246) at (20,2) {};
\node[cdot](247) at (21,3) {};\node[cdot](248) at (22,3) {};\node[cdot](249) at (23,3) {};\node[cdot](250) at (21,2) {};\node[cdot](251) at (22,2) {};\node[cdot](252) at (23,2) {};
\node[cdot](253) at (24,3) {};\node[cdot](254) at (25,3) {};\node[cdot](255) at (26,3) {};\node[cdot](256) at (24,2) {};\node[cdot](257) at (25,2) {};\node[cdot](258) at (26,2) {};
\node[cdot](259) at (27,3) {};\node[cdot](260) at (28,3) {};\node[cdot](261) at (29,3) {};\node[cdot](262) at (27,2) {};\node[cdot](263) at (28,2) {};\node[cdot](264) at (29,2) {};
\node[cdot](265) at (30,3) {};\node[cdot](266) at (31,3) {};\node[cdot](267) at (32,3) {};\node[cdot](268) at (30,2) {};\node[cdot](269) at (31,2) {};\node[cdot](270) at (32,2) {};
\node[cdot](271) at (33,3) {};\node[cdot](272) at (34,3) {};\node[cdot](273) at (35,3) {};\node[cdot](274) at (33,2) {};\node[cdot](275) at (34,2) {};\node[cdot](276) at (35,2) {};
\node[cdot](277) at (36,3) {};\node[cdot](278) at (37,3) {};\node[cdot](279) at (38,3) {};\node[cdot](280) at (36,2) {};\node[cdot](281) at (37,2) {};\node[cdot](282) at (38,2) {};
\node[cdot](283) at (39,3) {};\node[cdot](284) at (40,3) {};\node[cdot](285) at (41,3) {};\node[cdot](286) at (39,2) {};\node[cdot](287) at (40,2) {};\node[cdot](288) at (41,2) {};
\node[cdot](289) at (42,3) {};\node[cdot](290) at (43,3) {};\node[cdot](291) at (44,3) {};\node[cdot](292) at (42,2) {};\node[cdot](293) at (43,2) {};\node[cdot](294) at (44,2) {};
\node[cdot](295) at (45,3) {};\node[cdot](296) at (46,3) {};\node[cdot](297) at (47,3) {};\node[cdot](298) at (45,2) {};\node[cdot](299) at (46,2) {};\node[cdot](300) at (47,2) {};
\node[cdot](301) at (48,3) {};\node[cdot](302) at (49,3) {};\node[cdot](303) at (50,3) {};\node[cdot](304) at (48,2) {};\node[cdot](305) at (49,2) {};\node[cdot](306) at (50,2) {};

\draw (1)--(4)(2)--(5)(3)--(6);
\draw (7)--(10)(7)--(11)(8)--(11)(9)--(12);
\draw (13)--(16)(13)--(18)(14)--(17)(15)--(18);
\draw (19)--(22)(20)--(22)(20)--(23)(21)--(24);
\draw (25)--(28)(26)--(29)(26)--(30)(27)--(30);
\draw (31)--(34)(32)--(35)(33)--(34)(33)--(36);
\draw (37)--(40)(38)--(41)(39)--(41)(39)--(42);
\draw (43)--(46)(43)--(47)(43)--(48)(44)--(47)(45)--(48);
\draw (49)--(52)(49)--(53)(50)--(52)(50)--(53)(51)--(54);
\draw (55)--(58)(55)--(59)(56)--(59)(56)--(60)(57)--(60);
\draw (61)--(64)(61)--(65)(62)--(65)(63)--(64)(63)--(66);
\draw (67)--(70)(67)--(71)(68)--(71)(69)--(71)(69)--(72);
\draw (73)--(76)(73)--(78)(74)--(76)(74)--(77)(75)--(78);
\draw (79)--(82)(79)--(84)(80)--(83)(80)--(84)(81)--(84);
\draw (85)--(88)(85)--(90)(86)--(89)(87)--(88)(87)--(90);
\draw (91)--(94)(91)--(96)(92)--(95)(93)--(95)(93)--(96);
\draw (97)--(100)(98)--(100)(98)--(101)(98)--(102)(99)--(102);
\draw (103)--(106)(104)--(106)(104)--(107)(105)--(106)(105)--(108);
\draw (109)--(112)(110)--(112)(110)--(113)(111)--(113)(111)--(114);
\draw (115)--(118)(116)--(119)(116)--(120)(117)--(118)(117)--(120);
\draw (121)--(124)(122)--(125)(122)--(126)(123)--(125)(123)--(126);
\draw (127)--(130)(128)--(131)(129)--(130)(129)--(131)(129)--(132);
\draw (133)--(136)(133)--(137)(133)--(138)(134)--(137)(134)--(138)(135)--(138);
\draw (139)--(142)(139)--(143)(139)--(144)(140)--(143)(141)--(142)(141)--(144);
\draw (145)--(148)(145)--(149)(145)--(150)(146)--(149)(147)--(149)(147)--(150);
\draw (151)--(154)(151)--(155)(152)--(154)(152)--(155)(152)--(156)(153)--(156);
\draw (157)--(160)(157)--(161)(158)--(160)(158)--(161)(159)--(160)(159)--(162);
\draw (163)--(166)(163)--(167)(164)--(166)(164)--(167)(165)--(167)(165)--(168);
\draw (169)--(172)(169)--(173)(170)--(173)(170)--(174)(171)--(172)(171)--(174);
\draw (175)--(178)(175)--(179)(176)--(179)(176)--(180)(177)--(179)(177)--(180);
\draw (181)--(184)(181)--(185)(182)--(185)(183)--(184)(183)--(185)(183)--(186);
\draw (187)--(190)(187)--(192)(188)--(190)(188)--(191)(188)--(192)(189)--(192);
\draw (193)--(196)(193)--(198)(194)--(196)(194)--(197)(195)--(196)(195)--(198);
\draw (199)--(202)(199)--(204)(200)--(203)(200)--(204)(201)--(202)(201)--(204);
\draw (205)--(208)(205)--(210)(206)--(209)(206)--(210)(207)--(209)(207)--(210);
\draw (211)--(214)(212)--(214)(212)--(215)(212)--(216)(213)--(214)(213)--(216);
\draw (217)--(220)(218)--(220)(218)--(221)(218)--(222)(219)--(221)(219)--(222);
\draw (223)--(226)(224)--(226)(224)--(227)(225)--(226)(225)--(227)(225)--(228);
\draw (229)--(232)(229)--(233)(229)--(234)(230)--(232)(230)--(233)(230)--(234)(231)--(234);
\draw (235)--(238)(235)--(239)(235)--(240)(236)--(239)(236)--(240)(237)--(238)(237)--(240);
\draw (241)--(244)(241)--(245)(241)--(246)(242)--(245)(242)--(246)(243)--(245)(243)--(246);
\draw (247)--(250)(247)--(251)(247)--(252)(248)--(251)(249)--(250)(249)--(251)(249)--(252);
\draw (253)--(256)(253)--(257)(254)--(256)(254)--(257)(254)--(258)(255)--(256)(255)--(258);
\draw (259)--(262)(259)--(263)(260)--(262)(260)--(263)(260)--(264)(261)--(263)(261)--(264);
\draw (265)--(268)(265)--(269)(266)--(268)(266)--(269)(267)--(268)(267)--(269)(267)--(270);
\draw (271)--(274)(271)--(276)(272)--(274)(272)--(275)(272)--(276)(273)--(274)(273)--(276);
\draw (277)--(280)(278)--(280)(278)--(281)(278)--(282)(279)--(280)(279)--(281)(279)--(282);
\draw (283)--(286)(283)--(287)(283)--(288)(284)--(286)(284)--(287)(284)--(288)(285)--(286)(285)--(288);
\draw (289)--(292)(289)--(293)(289)--(294)(290)--(292)(290)--(293)(290)--(294)(291)--(293)(291)--(294);
\draw (295)--(298)(295)--(299)(296)--(298)(296)--(299)(296)--(300)(297)--(298)(297)--(299)(297)--(300);
\draw (301)--(304)(301)--(305)(301)--(306)(302)--(304)(302)--(305)(302)--(306)(303)--(304)(303)--(305)(303)--(306);
\end{tikzpicture}
\end{center}
\vspace{10pt}
Here, we also assumed that the collection of minimal elements is the unique antichain containing $4$ elements.
Similarly, there are $13$ configurations for the covers between rank $11$ and $12$, $4$ configurations for the covers in ranks $12$ to $19$, and only $1$ configuration for the ranks $19$ to $29$.

Let $d$ be the highest degree, for instance $d=29$ in type $H_4$.

\algo{EnumeratePosets2}{$\Gamma'$,$p$}{Enumerates all root posets starting from $\Gamma'$}
{A poset $\Gamma'$, an index $p$ for the next degree.}
{Posets based upon $\Gamma'$ satisfying Properties~\ref{prop:parabolicsubsystems} through~\ref{prop:antichainorbitsize}, except for the standard parabolic substructure.}
{
\item If $p=d$ then:
\begin{enumerate}
\item Count the antichains with $2$ or $3$ elements which do not contain a simple root. If this number wrong, return $0$.
\item Compute orbits under the Panyushev action. As soon as one of them has a forbidden length, return $0$.
\item Print $\Gamma$ to an output file, return $1$.
\end{enumerate}
\item If all elements of the last added degree are greater than all simple roots, then count the antichains with $2$ or $3$
elements which contain simple roots. These numbers will not change anymore, so we may return $0$ if they are wrong.
\item Count antichains and use Property~\ref{prop:antichainsize} as upper bounds, return $0$ if one of them is violated.
\item Compute orbits under the Panyushev action. As soon as one of them has a forbidden length, return $0$.
\item For all possible configurations of covers for the next degree:
\begin{enumerate}
\item Include the next level to $\Gamma'$ to a new poset $\Gamma''$,
\item Call {\bf EnumeratePosets2}($\Gamma'$,$p+1$).
\end{enumerate}
}

\begin{remar} One advantage of Algorithm~\ref{EnumeratePosets2} is that it is independent of the structure of $H_4$.
Indeed, we have tested it as well with the root posets of types $H_3$, $B_4$, and $F_4$ (and, of course, adjusted combinatorial information). It turns out that (up to symmetries) these posets are uniquely determined by these properties.
\end{remar}

\subsubsection{Results} A slight modification of Algorithm~\ref{EnumeratePosets2} produces all posets with Properties~\ref{prop:parabolicsubsystems} through~\ref{prop:antichainsize} and satisfying as well Property~\ref{prop:antichainorbitsize}(a). There are $4$ such posets (see \fref{fig:rp12345_2}), none of them satisfies Property~\ref{prop:antichainorbitsize}(b).
We do indeed think that the Panyushev action does simply not behave well in type $H_4$.

\subsection{Assuming as well Property~\ref{prop:antichainorderidealsize}}
\label{sec:H4algorithms2}

In Section \ref{sec:hilbertseries}, we discuss the computations that have led to a conjectured Hilbert series of the diagonal coinvariants which should possibly provide the numbers of order ideals of given sizes (Property~\ref{prop:antichainorderidealsize}). A slight modification of Algorithm~\ref{EnumeratePosets2} is capable of enumerating all posets with Properties~\ref{prop:parabolicsubsystems} through~\ref{prop:antichainsize} and Property~\ref{prop:antichainorderidealsize}. Notice that Property~\ref{prop:antichainorderidealsize} is such a strong condition that the computation just takes a few seconds this time.
It turns out that there is no poset agreeing with the polynomial $\hilbert_{H_4}(q,t)$ conjectured in Conjecture~\ref{conj:hilbertpol}.

\subsubsection{Assuming Conjecture~\ref{conj:hilbertpol} to be false}
  Even though we strongly believe the Hilbert series in Conjecture~\ref{conj:hilbertpol} to be true, one might ask how the situation looks like if this conjecture turns out to be false.
  In particular, if it could then be possible to find a poset satisfying Properties~\ref{prop:parabolicsubsystems} through~\ref{prop:antichainsize} and Property~\ref{prop:antichainorderidealsize}.
  To this end, assume now that there exists such a poset, and that the Hilbert series $\hilbert_{H_4}(q,t)$ satisfies assumptions~\eqref{item:assumption1} and~\eqref{item:assumption2}, and that $a_{i,n}\ne 0$ implies that $q^*[n]_{q^2}$ appears in the decomposition of $q^{60}\hilbert_{H_4}(q,q^{-1})$ (see Section~\ref{sec:furtherassumptions} for the definition of the two assumptions, of the decomposition of $q^{60}\hilbert_{H_4}(q,q^{-1})$, and of $a_{i,n}$).
  By the well behavedness with respect to standard parabolic subgroups in Property~\ref{prop:parabolicsubsystems}, the terms of low degree of $\hilbert_{H_4}(q,1)$ are given by
  \[ 1+4q+6q^2+7q^3+8q^4+8q^5+9q^6+8q^7+8q^8+8q^9+9q^{10}+\ldots, \]
  and the terms of high degree are given by
  \[ \ldots+2q^{49}+q^{50}+q^{51}+\ldots+q^{60}. \]
  Then one can compute that there are $180$
  possible polynomials for $\hilbert_{H_4}(q,t)$ with these properties, which moreover all have different specializations $t=1$.
  Finally, there are $120$ posets satisfying Properties~\ref{prop:parabolicsubsystems} through~\ref{prop:antichainsize} and whose order ideals yield one of these specializations. Only $7$ of the $180$ different possible polynomials $\hilbert_{H_4}(q,t)$ occur.
  The following list gives the $7$ polynomials and the number of corresponding posets.
\begin{align*}
2  \quad\times\quad& \phantom{q^0t^0}[61]_{q,t}+q^1t^1[49]_{q,t}+q^{3}t^{3}[41]_{q,t}+q^1t^1[37]_{q,t}+q^{4}t^{4}[31]_{q,t}+\\
  &q^{2}t^{2}[25]_{q,t}+q^1t^1[21]_{q,t}+q^{2}t^{2}[13]_{q,t}+q^{6}t^{6}\phantom{[\;1\;]_{q,t}}+q^{10}t^{10}\\[15pt]
10 \quad\times\quad& \phantom{q^0t^0}[61]_{q,t}+q^1t^1[49]_{q,t}+q^1t^1[41]_{q,t}+q^{3}t^{3}[37]_{q,t}+q^1t^1[31]_{q,t}+\\
  &q^{4}t^{4}[25]_{q,t}+q^{2}t^{2}[21]_{q,t}+q^{2}t^{2}[13]_{q,t}+q^{6}t^{6}\phantom{[\;1\;]_{q,t}}+q^{10}t^{10}\\[15pt]
12 \quad\times\quad& \phantom{q^0t^0}[61]_{q,t}+q^1t^1[49]_{q,t}+q^1t^1[41]_{q,t}+q^{4}t^{4}[37]_{q,t}+q^1t^1[31]_{q,t}+\\
  &q^{3}t^{3}[25]_{q,t}+q^{2}t^{2}[21]_{q,t}+q^{2}t^{2}[13]_{q,t}+q^{6}t^{6}\phantom{[\;1\;]_{q,t}}+q^{10}t^{10}\\[15pt]
16 \quad\times\quad& \phantom{q^0t^0}[61]_{q,t}+q^1t^1[49]_{q,t}+q^1t^1[41]_{q,t}+q^{4}t^{4}[37]_{q,t}+q^1t^1[31]_{q,t}+\\
  &q^{2}t^{2}[25]_{q,t}+q^{2}t^{2}[21]_{q,t}+q^{3}t^{3}[13]_{q,t}+q^{6}t^{6}\phantom{[\;1\;]_{q,t}}+q^{10}t^{10}\\[15pt]
20 \quad\times\quad& \phantom{q^0t^0}[61]_{q,t}+q^1t^1[49]_{q,t}+q^{3}t^{3}[41]_{q,t}+q^1t^1[37]_{q,t}+q^1t^1[31]_{q,t}+\\
  &q^{4}t^{4}[25]_{q,t}+q^{2}t^{2}[21]_{q,t}+q^{2}t^{2}[13]_{q,t}+q^{6}t^{6}\phantom{[\;1\;]_{q,t}}+q^{10}t^{10}\\[15pt]
20 \quad\times\quad& \phantom{q^0t^0}[61]_{q,t}+q^1t^1[49]_{q,t}+q^1t^1[41]_{q,t}+q^{3}t^{3}[37]_{q,t}+q^1t^1[31]_{q,t}+\\
  &q^{2}t^{2}[25]_{q,t}+q^{2}t^{2}[21]_{q,t}+q^{4}t^{4}[13]_{q,t}+q^{6}t^{6}\phantom{[\;1\;]_{q,t}}+q^{10}t^{10}\\[15pt]
40 \quad\times\quad& \phantom{q^0t^0}[61]_{q,t}+q^1t^1[49]_{q,t}+q^{3}t^{3}[41]_{q,t}+q^1t^1[37]_{q,t}+q^1t^1[31]_{q,t}+\\
  &q^{2}t^{2}[25]_{q,t}+q^{2}t^{2}[21]_{q,t}+q^{4}t^{4}[13]_{q,t}+q^{6}t^{6}\phantom{[\;1\;]_{q,t}}+q^{10}t^{10}
\end{align*}

\section{Computation of the Hilbert series of the diagonal coinvariants}
\label{sec:hilbertseries}

To compute the Hilbert series $\hilbert_{H_4}(q,t) = \hilbert\big( \coinv(H_4);q,t \big)$ in Property~\ref{prop:antichainorderidealsize}, we must compute a minimal basis~$B$
of the ideal in $\CC[V\oplus V]$ generated by the elements
\[ \theta(m) := \sum_{w\in W} \det(w)\: w(m) \]
for all monomial $m = x_1^{\alpha_1} \cdots x_4^{\alpha_4} y_1^{\beta_1} \cdots y_4^{\beta_4} \in \CC[\x,\y]$. Here, $W = W(H_4)$ is the Coxeter group of type $H_4$ acting diagonally on $\CC[\x,\y]$.
Each element $f\in B$ is homogeneous in $\x$ and in $\y$ of bidegree $(\degx,\degy)$, and the Hilbert series is then given by
\[ \hilbert_{H_4}(q,t) = \sum_{f\in B} q^{\degx} t^{\degy}. \]

In this process, we face the following two problems.
First, there is the problem of computing $\theta(m)$ for a given monomial $m \in \CC[\x,\y]$. This computation is closely related to computing the Reynolds operator in classical computational invariant theory (over polynomial rings) since $W$ has a subgroup of index $2$ in which all elements have determinant $1$. However, computing invariants for a group of size $7200$ acting on an $8$-dimensional space is slightly out of the range of the classical algorithms.
Second, there is the problem that computations in fields of coefficients that can be used to provide faithful $4$-dimensional turn out to be too slow for the needed Gröbner basis computations.

\medskip

To work around these problems, we are forced to resort to some heuristics.
We present three steps which accelerate the computation considerably, two of them are assumptions under which we cannot prove anymore that the computed series is indeed correct.

\subsection{Finding a nice representation}
The largest abelian subgroup $U$ in $W$ has $50$ elements. We choose a representation of $W$ in which these $50$ elements are diagonal matrices. Luckily, for this choice of a basis, $W$ has a subgroup $U'$ of $400$ elements which are all monomial matrices.
Thus we compute $\theta(m)$ in the following way.

\algo{Theta}{$m$}{Apply the map $\theta$}
{A monomial $m\in \CC[x_1,\ldots,x_4,y_1,\ldots,y_4]$.}
{$\theta(m)$.}
{\label{algo:theta}
\item Compute $z:=\sum_{w\in U} \det(w)\: w(m)$. This is very easy, since all elements of $U$ are diagonal matrices.
\item If $z\ne 0$, then:
\begin{enumerate}
  \item Compute $z':=\sum_{wU\in U'/U} \det(w)\: w(m)$. Computing each of the summands is just evaluation, since all elements of $U'$ are monomial matrices.
  \item If $z'\ne 0$, then return $g:=\sum_{wU'\in W/U'} \det(w)\: w(z')$.
\end{enumerate}
\item Return $0$.
}

\begin{remar}
  The first sum has $50$ summands, the second one $8$, and the last sum has $36$ summands. Thus instead of computing $14400$ times $w(m)$, we only need $36$ of these expensive computations, and we only perform them if $z$ and $z'$ are not zero.
  In practice, before returning $g$ we also evaluate $g$ at $(y_1,\ldots,y_4,x_1,\ldots,x_4)$ to avoid a second computation of this invariant.
\end{remar}

\subsection{Choosing a fast field of coefficients}
In principle, faithful $4$-dimensional representations of $W$ require a field containing the golden ratio, for instance $\QQ(\sqrt{5})$.
To apply Algorithm~\ref{algo:theta}, we need fifth roots of unity because we need to diagonalize simultaneously the subgroup $U$.
This is not a problem since implementations of cyclotomic fields are highly optimized in most computer algebra systems.
However, although they are quite fast, these fields (in characteristic zero) are too slow for the last step, the computation of the Gr\"obner basis.
We therefore perform all computations over a finite field $\FF_q$ of order $q$ with $5\mid (q-1)$ (so that the $5$-th roots of unity are included). This can result in a slightly different Hilbert series, but with large $q$ this is very unlikely.

\subsection{Reducing the number of variables}
Finally, since we are only interested in the degrees of the components of the elements of $B$, we may evaluate $y_1,\ldots,y_4$ at
$z_1t,\ldots,z_4t$ for some $z_1,\ldots,z_4\in\FF_q$ \emph{before} computing $B$. This could lead to some wrong coefficients in
$\hilbert_{H_4}(q,t)$. But without this simplification, the computation of $B$ is not feasible on modern computers.

\subsection{Further assumptions on \texorpdfstring{$\hilbert\big( \coinv(\rootSystem); q, t \big)$}{the Hilbert series}}
\label{sec:furtherassumptions}
Even though we already have made concessions, the above techniques are not quite sufficient to obtain all the required coefficients.
We finally fill the gaps in the computed series by using the following further assumptions.
\begin{enumerate}
  \item The series $\hilbert\big( \coinv(\rootSystem); q, t \big)$ is of the form
  \[ \sum_{i\in\NN, n\in\NN} a_{i,n} q^i t^i [n]_{q,t} \]
  for some $a_{i,n}\in\NN$, $i,n\in\NN$. \label{item:assumption1}

  \item The evaluation $t=q^{-1}$ in the Hilbert series yields\footnote{This assumption is based on~\cite[Conjecture~3]{Stu2010}.}
  \[  q^{60}\hilbert_{H_4}(q,q^{-1})
  = \frac{[32]_q[42]_q[50]_q[60]_q}{[2]_q[12]_q[20]_q[30]_q}, \]
  \label{item:assumption2}
  where we write $[n]_q = [n]_{q,1} = q^{n-1} + q^{n-2} + \ldots + 2 + 1$.
\end{enumerate}
Under these assumptions, the above heuristics, and with extensive computations we obtain the following conjecture.

\begin{conje}\label{conj:hilbertpol}
The Hilbert series of the diagonal coinvariants of the Coxeter group of type $H_4$ is given by
\begin{align*}
  \hilbert_{H_4}(q,t) =\:\:& [61]_{q,t} + qt[49]_{q,t} + qt[41]_{q,t} + q^2t^2[37]_{q,t} + qt[31]_{q,t} +\\
                           & q^3t^3[25]_{q,t} q^2t^2[21]_{q,t} + q^4t^4[13]_{q,t} + q^6t^6 + q^{10}t^{10}.
\end{align*}
\end{conje}

Notice that the polynomial $q^{60}\hilbert_{H_4}(q,q^{-1})$ in assumption~\eqref{item:assumption2} has the decomposition
\begin{align*}
  q^{60}\hilbert_{H_4}(q,q^{-1}) =\:\:& [61]_{q^2} + q^{12}[49]_{q^2}+q^{20}[41]_{q^2}+q^{24}[37]_{q^2}+q^{30}[31]_{q^2}+ \\
                                      & q^{36}[25]_{q^2}+q^{40}[21]_{q^2}+q^{48}[13]_{q^2}+2q^{60},
\end{align*}
and that this decomposition in such summands is unique in the sense that consecutive summands $q^a[b]_{q^2}$ and $q^c[d]_{q^2}$ satisfy $a<c$ and $b>d$. One can now observe that the sequence
$$61,49,41,37,31,25,21,13,1,1$$
of numbers involved in this decomposition coincides with the sequence of numbers involved in the conjectured Hilbert series.

\end{document}